\documentclass[11pt, a4paper,leqno]{amsart}
\usepackage{latexsym,amsmath,amssymb,amsthm}
\usepackage[T1]{fontenc}
\usepackage[utf8]{inputenc}
\usepackage{color}

\title[Uhlenbeck's  decomposition]{Uhlenbeck's decomposition in Sobolev and Morrey-Sobolev spaces}

\author{Pawe\l{} Goldstein}
\address{Pawe\l{} Goldstein\newline \indent Faculty of Mathematics, Informatics and Mechanics\newline \indent University of Warsaw}

\author{Anna Zatorska-Goldstein}
\address{Anna Zatorska-Goldstein\newline \indent Faculty of Mathematics, Informatics and Mechanics\newline \indent University of Warsaw}
\thanks{The research has been supported by NCN grant SONATA BIS no.~2012/05/E/ST1/03232 and by the Foundation for Polish Science grant no. POMOST BIS/2012-6/3.}

\newtheorem{theorem}{Theorem}[section]
\newtheorem{lemma}[theorem]{Lemma}
\newtheorem{corollary}[theorem]{Corollary}
\newtheorem{proposition}[theorem]{Proposition}
\newtheorem{definition}[theorem]{Definition}

\newtheorem*{theorem*}{Theorem}
\theoremstyle{definition}
\newtheorem{remark}[theorem]{Remark}
\newtheorem*{remarknn}{Remark}

\newtheorem*{question}{Question}
\numberwithin{equation}{section}

\newcommand{\vertiii}[1]{{\left\vert\kern-0.25ex\left\vert\kern-0.25ex\left\vert #1
    \right\vert\kern-0.25ex\right\vert\kern-0.25ex\right\vert}}

\def\mvint_#1{\mathchoice
          {\mathop{\vrule width 6pt height 3 pt depth -2.5pt
                  \kern -9pt \intop}\nolimits_{\kern -3pt #1}}%
          {\mathop{\vrule width 5pt height 3 pt depth -2.6pt
                  \kern -6pt \intop}\nolimits_{#1}}%
          {\mathop{\vrule width 5pt height 3 pt depth -2.6pt
                  \kern -6pt \intop}\nolimits_{#1}}%
          {\mathop{\vrule width 5pt height 3 pt depth -2.6pt
                  \kern -6pt \intop}\nolimits_{#1}}}
\newcommand{\R}{\mathbb{R}}
\newcommand{\bb}{\mathbb{B}}
\newcommand{\bs}{\mathcal{B}}

\newcommand{\id}{\mathrm{Id}}
\begin{document}
\begin{abstract}
We present a self-contained proof of Rivi\`ere's theorem on the existence of Uhlenbeck's decomposition for  $\Omega\in L^p(\bb^n,so(m)\otimes\Lambda^1\R^n)$ for $p\in (1,n)$, with Sobolev type estimates in the case $p \in[n/2,n)$ and Sobolev-Morrey type estimates in the case $p\in (1,n/2)$.  We also prove an analogous theorem in the case when $\Omega\in L^p( \bb^n, TCO_{+}(m) \otimes \Lambda^1\R^n)$, which corresponds to Uhlenbeck's decomposition with conformal gauge group.
\end{abstract}
\maketitle

\section{Introduction}
Throughout the paper, $n \geq 2$.

 In 2006 Tristan Rivi\`ere published a solution to the so-called Heinz-Hilde\-brandt conjecture on regularity of solutions to conformally invariant nonlinear systems of partial differential equations in dimension 2 \cite{Riviere1}. The key tool he used was a theorem due to Karen Uhlenbeck, on the existence of the so-called Coulomb gauges, in which the connection of a line bundle takes particularly simple form; we quote the theorem below in somewhat imprecise terms, to avoid unnecessary technicalities.
\begin{theorem}[Uhlenbeck, \cite{Uhlenbeck}] Let $\eta$ be a vector bundle over the unit ball $\bb^n$, $n\geq 2$, with connection form $\Omega\in W^{1,p}$, $p\in [n/2,n)$. If the curvature field, $F(\Omega)=\Omega\wedge\Omega+d\Omega$, has sufficiently small $L^{n/2}$ norm, then $\Omega$ is gauge-equivalent to a connection $\tilde{\Omega}$ which is co-closed (i.e. $d\ast\tilde{\Omega}=0$), with estimates of the gauge and of $\tilde{\Omega}$ given in terms of $\|F(\Omega)\|_{L^p}$.
\end{theorem}
This theorem was later generalized, to allow applications in higher dimensions ($n>4$), to connections with smallness condition on the curvature given in Morrey norms (\cite{MeyerRiviere},\cite{TaoTian}).

We should mention that Coulomb gauges appeared in the theory of geometrically motivated systems of PDE in important papers of F. H\'elein on regularity of harmonic mappings between manifolds (\cite{Helein1,Helein2}).

The power of Rivi\`ere's idea was in the fact that he used Uhlenbeck's theorem to antisymmetric differential forms, which \textit{a priori} were not interpreted as connection forms, even if the problem had clear geometric motivation. Moreover, he reformulated the theorem in a language more suited for PDE applications.
Simplifying, Rivi\`ere's theorem (see \cite[Lemma A.3]{Riviere1}) says that any antisymmetric matrix $\Omega$ of $1$-forms on a ball
$$
\Omega : \bb^n \to so(m) \otimes \Lambda^1 \R^n
$$
with \textit{sufficiently small} norm can be transformed by an orthogonal change of coordinates (gauge transformation)
$$
P: \bb^n \to SO(m)
$$
to an antisymmetric matrix of co-closed forms (up to a rather regular term)
$$
P^{-1}dP+P^{-1}\Omega P = \ast d\xi,
$$
$$
\text{i.e. } \Omega = P (\ast d\xi) P^{-1} - dP \,P^{-1} .
$$
Here $\xi$ is an antisymmetric matrix of $(n-2)$-forms:
$$
\xi: \bb^n \to so(m)\otimes \Lambda^{n-2}\R^n.
$$
Such a decomposition of $\Omega$ is often referred to as Uhlenbeck's decomposition.

One should note that, in contrast to Uhlenbeck's theorem, the smallness condition is imposed on $\Omega$, and not on $F(\Omega)$. This has the advantage of being a simpler and more natural expression in the scope of general PDE's, but, at the same time, it is less natural in the scope of gauge theory, since the norm of the curvature is gauge independent.

Starting with Rivi\`ere's result, Uhlenbeck's decomposition appeared in numerous papers on nonlinear PDE's and variational problems, each time adapted to a specific system, function spaces and dimensions, and with different smallness conditions:
\begin{itemize}
\item Rivi\`ere \cite{Riviere1}: $\Omega \in L^2$, $n=2$;
\item Rivi\`ere \& Struwe \cite{RiviereStruwe}: $\Omega \in L^{2,n-2}$, $n>2$;
\item Lamm \& Rivi\`ere \cite{LammRiviere}: $\Omega \in W^{1,2}$, $n=4$;
\item Meyer \& Rivi\`ere \cite{MeyerRiviere} and Tao \& Tian \cite{TaoTian}:
$$
\Omega \in W^{1,2} \cap L^{4,n-4}, \quad d \Omega \in L^{2,n-4} \qquad \text{for $n \geq 4$};
$$
\item M\"uller \& Schikorra \cite{MuellerSchikorra}: $\Omega \in W^{1,2}$, $n=2$;
\end{itemize}

All the proofs of the above results are, up to details, adaptations of the original approach of Uhlenbeck and most of them refer the reader for certain parts of the reasoning to the original paper \cite{Uhlenbeck}.  The latter, however, is written in the language of differential geometry (the result was used there in the context of the existence theory for Yang-Mills fields). Translating the results to Rivi\`ere's setting and filling all the sketched details was not trivial, which is probably why this extremely useful result went overlooked by the PDE community for over two decades.

\medskip

All the proofs naturally split into  two parts:
\begin{itemize}
\item proving the existence of the decomposition for any sufficiently small perturbation of a co-closed form $\ast d \zeta$, provided certain norm of $d\zeta$ is small;
\item proving that once we have $\Omega$, $P$ and $\xi$ which satisfy the equation
$$
P^{-1}dP+P^{-1}\Omega P = \ast d\xi
$$
and additionally certain norms of $\Omega$, $dP$ and $d\xi$ are sufficiently small, the presumed estimates hold (the norms of $P$ and $\xi$ are bounded in terms of the norm of $\Omega$).
\end{itemize}

In the results mentioned above, two strategies of proving the existence of decomposition of $\Omega$ were used. The original strategy used by Karen Uhlenbeck was to solve the equation for $P$
$$
\ast d \ast (P^{-1} dP + P^{-1}(\ast d \zeta + \lambda)P) = 0
$$
for a given perturbation $\lambda$ of some fixed co-closed form. To do this, we look for $P$ of the form $P=e^u$, add some boundary condition on $u$ (Neumann in the original result of Uhlenbeck, Cauchy in our proof) and define the nonlinear operator
$$
T(u,\lambda) = \ast d \ast (e^{-u} de^u + e^{-u}(\ast d \zeta + \lambda)e^u)
$$
acting on appropriate Banach spaces; in our case
$$
T(u,\lambda): \mathcal{B} \times W^{1,p} \to L^p
$$
where $\mathcal{B} = W^{2,p}\cap W^{1,p}_o$. One can apply then the Implicit Function Theorem to show that  $T(u,\lambda) = 0$ has a solution $u_\lambda$ continuously depending on $\lambda$.
To do this, one has to show that the linearization of  $T$ at $(u,\lambda) = (0,0)$ with respect to the first argument,
$$
H(u) = \Delta u + \ast [d \zeta, du],
$$
is an isomorphism $\mathcal{B} \to L^p$.
This strategy works in Sobolev spaces for ${p\in(n/2;n)}$, but it fails when $1<p < n/2$.

Another strategy is used by T. Tao and G. Tian in \cite{TaoTian}. Again, one looks for $P=e^u$ and assumes $u$ has zero boundary data. The equation
$$
\ast d \ast (P^{-1} dP + P^{-1}(\ast d \zeta + \lambda)P) = 0
$$
is transformed into the form
$$
\Delta u = \ast d \ast
\left( du - e^{-u} de^u - e^{-u}(\ast d \zeta + \lambda) e^u \right)=\ast d \ast F(u,\zeta,\lambda).
$$
Then, an iteration scheme is set to provide a solution:
\begin{equation*}
\left\{\begin{aligned}u^0 &= 0 \\
\Delta u^{k+1} &= \ast d \ast F(u^k,\zeta, \lambda).
\end{aligned}\right.
\end{equation*}
We use this strategy when $1<p < n/2$.

In 2009 A. Schikorra gave an alternative, variational proof of the existence of
Uhlenbeck's decomposition (\cite{armin}). His approach was inspired by a similar variational construction of a moving frame by F. H\'elein \cite{Helein2}. Schikorra's methods, however, provided the gauge transformation $P$ only in $W^{1,2}$, even for $\Omega \in L^p$ with $p>2$
(while Uhlenbeck's and Rivi\`ere's approach gave $P\in W^{1,p}$). On the other hand, his method was much simpler and allowed him to give alternative regularity
proofs for systems studied by Rivi\`ere \cite{Riviere1} and Rivi\`ere \& Struwe \cite{RiviereStruwe}.

Finally, one should mention the book of K. Wehrheim \cite{wehrheim}, who undertook the effort of clarifying and presenting in all detail the original result of K.~Uhlenbeck.
\medskip

The main result of the paper is a self-contained, complete proof of Rivi\`ere's theorem in the following settings.

First, with Sobolev type estimates,
\begin{theorem}\label{thm main}
Let $\frac{n}{2}\leq p<n$. There exists $\epsilon >0$ such that for any antisymmetric matrix $\Omega \in W^{1,p}(\bb^n, so(m)\otimes \Lambda^1\R^n)$  of $1$-differential forms on $\bb^n$ such that
$$
\|\Omega \|_{L^n}<\epsilon
$$
there exist $P\in W^{2,p}(\bb^n,SO(m))$, and $\xi \in W^{2,p}\cap W_0^{1,p}(\bb^n, so(m)\otimes \Lambda^{n-2}\R^n)$ satisfying the system
\begin{equation*}
\left\{
\begin{aligned}
P^{-1}dP+P^{-1}\Omega P&=\ast d\xi &&\text{on }\bb^n,\\
d\ast\xi&=0 &&\text{on }\bb^n,\\
\xi&=0 &&\text{on } \partial \bb^n, \\
P&=\mathrm{Id} &&\text{on } \partial \bb^n;
\end{aligned}
\right.
\end{equation*}
and such that
\begin{align*}
\|d\xi \|_{L^n}+\|dP\|_{L^n}&\leq C(n,m)\|\Omega \|_{L^n},\\
\|d\xi\|_{W^{1,p}}+\|dP\|_{W^{1,p}}&\leq C(n,m,p)\|\Omega \|_{W^{1,p}},\\
\text{ and, if }p>n/2,\qquad \|d\xi \|_{L^p}+\|dP\|_{L^p}&\leq C(n,m,p)\|\Omega \|_{L^p}.
\end{align*}
\end{theorem}
Next, with Morrey-Sobolev type estimates,
\begin{theorem}\label{u-m}
Assume $1< p < n/2$. Let
$$
\Omega : \bb^n \to so(m)\otimes \Lambda^1\R^n
$$
be an antisymmetric matrix of $1$-differential forms on $\bb^n$. Assume
$$
\Omega \in L^{2p,n-2p} \quad \text{and} \quad d\Omega \in L^{p,n-2p}
$$
There exists $\epsilon >0$ such that if $\Omega$ satisfies the smallness condition
$$
\|\Omega \|_{L^{2p,n-2p}}<\epsilon
$$
then there exist $P\in W^{2,p}(\bb^n,SO(m))$ and $\xi \in W^{2,p}(\bb^n, so(m)\otimes \Lambda^{n-2}\R^n)$ satisfying the system
\begin{equation*}
\left\{
\begin{aligned}
P^{-1}dP+P^{-1}\Omega P&=\ast d\xi &&\text{on }\bb^n,\\
d\ast\xi&=0 &&\text{on }\bb^n,\\
\xi&=0 &&\text{on } \partial \bb^n, \\
P&=\mathrm{Id} &&\text{on } \partial \bb^n.
\end{aligned}
\right.
\end{equation*}
Moreover $P, \xi \in L^{p, n-2p}_2$ with
\begin{align}
\|d\xi \|_{L^{p, n-p}}+\|dP\|_{L^{p,n-p}}& \leq  C(n,m)\|\Omega \|_{L^{2p, n-2p}} \\
\|\Delta \xi \|_{L^{p, n-2p}}+\| \Delta P\|_{L^{p,n-2p}}& \leq
C(n,m)( \|\Omega \|_{L^{2p, n-2p}} + \|d\Omega \|_{L^{p, n-2p}}).
\end{align}
\end{theorem}

And finally a version of Uhlenbeck's decomposition for a larger gauge group $CO_+(n)$ (i.e. conformal
transformations), which gives the decomposition theorem for a larger class of matrix-valued differential forms.

\begin{theorem}\label{thm:conformal}
Let $\frac{n}{2}<p<n$. There exists $\epsilon >0$ such that for any $\Omega \in W^{1,p}(\bb^n, TCO_+(m)\otimes \Lambda^1\R^n)$ such that $\|\Omega \|_{L^n}<\epsilon $ there exist $S\colon\bb^n\to CO_+(m)$ satisfying $\ln |S|\in W^{2,p}(\bb^n)$, $S/|S|\in W^{2,p}(\bb^n, SO(m))$ and\\ ${\zeta \in W^{2,p}(\bb^n, TCO_+(m)\otimes \Lambda^{n-2}\R^n)}$ such that
\begin{equation}
\left\{
\begin{aligned}
S^{-1}dS+S^{-1}\Omega S&=\ast d\zeta &&\text{on }\bb^n,\\
d\ast\zeta&=0 &&\text{on }\bb^n,\\
\zeta&=0 &&\text{on } \partial \bb^n;
\end{aligned}
\right.
\end{equation}
and such that
\begin{align*}
\|d\zeta\|_{W^{1,p}}+\|d(S/|S|)\|_{W^{1,p}}+\|d\ln|S|\|_{W^{1,p}}&\leq C(n,m)\|\Omega \|_{W^{1,p}}\\
\|d\zeta \|_{L^p}+\|d(S/|S|)\|_{L^p}+\|d\ln |S|\|_{L^p}&\leq C(n,m)\|\Omega \|_{L^p},\\
\|d\zeta \|_{L^n}+\|d(S/|S|)\|_{L^n}+\|d\ln |S|\|_{L^n}&\leq C(n,m)\|\Omega \|_{L^n}.
\end{align*}
\end{theorem}
In view of Theorem~\ref{thm:conformal}, we may ask a natural question:

\begin{question} What is the largest Lie subgroup $G$ of $GL(n)$ that can be used in an analogue of Rivi\`ere's theorem: for any matrix of $1$-differential forms $\Omega \in T_{\id}G$ there exists a gauge transformation $P\in G$ such that $\Omega^P=P^{-1}dP+P^{-1}\Omega P$ is co-closed  and certain integrability estimates on $P$, $\Omega^P$ and their derivatives, in terms of $\Omega$, hold?
\end{question}
The paper is structured as follows. In Section~\ref{sec:forms}, we recall Gaffney's inequality and discuss, how the boundary conditions on the decomposition components $P$ and $\xi$ allow us to estimate their Sobolev norms with only some of their derivatives.

Next, in Section~\ref{sec:Morrey} we recall the definitions and basic properties of Morrey and Morrey-Sobolev spaces we use.

In Section~\ref{sec:Sobolev}, we prove Theorem~\ref{thm main}. Next, in Section~\ref{sec:MorreySobolev}, we prove Theorem~\ref{u-m}, and finally, in Section~\ref{sec: conf}, we prove Theorem~\ref{thm:conformal}.

Throughout the paper, wherever applicable, we use the \emph{operator} norm $|T|=\sup_{|x|=1} |Tx|$ of a linear operator $T$ (thus, in particular, $|P|=1$ almost everywhere) -- this simplifies the estimates of compositions. A constant $C$ may vary from line to line in calculations. 

\section{Sobolev spaces and differential forms} \label{sec:forms}
Throughout the paper, we use differential forms with coefficients in Sobolev spaces, i.e. Sobolev differential forms. With this in mind, we write e.g.
$$\Omega\in W^{1,p}(\bb^n,so(m)\otimes \Lambda^1\R^n),$$
which means that $\Omega$ is a $W^{1,p}$ function with values in the vector space $so(m)\otimes \Lambda^1\R^n$. This allows us to define the \emph{full} Sobolev norm $\|\Omega\|_{W^{1,p}}$, which disregards the differential form aspect of it.

It is tempting to consider Sobolev spaces of differential forms using only the two derivatives that are natural in this setting: the differential $d$ and the co-differential $\ast d \ast$, instead of the full derivative $D$. This is possible if we restrict to forms that satisfy certain boundary conditions, see e.g. \cite{IwaniecScottStroffolini}. It is important to realize, however, that for a general differential form $\omega$, the differential $d\omega$ and the co-differential $\ast d\ast \omega$ capture only some derivatives of the coefficients, and in general one should not expect to control the $W^{1,p}$-norm of $\omega$ by the $L^p$-norms of $\omega$, $d\omega$ and/or $\ast d \ast \omega$ (we may skip the first star and use $d\ast\omega$, since the Hodge star $\ast$ is an isomorphism and $\|\ast d\ast \omega\|_{L^p}=\|d\ast \omega\|_{L^p}$).

However, if the domain in which we consider our forms is smooth (as is in our case, where the only domains considered are $n$-dimensional balls) and the form $\omega$ satisfies certain boundary conditions (namely, $\omega$ has vanishing tangent or normal component on the domain's boundary), then Gaffney's inequality (\cite{Gaffney}, see also \cite[Theorem 4.8]{IwaniecScottStroffolini},\cite[Theorem 10.4.1]{IwaniecMartin}) holds:
\begin{equation}\label{gaffney}
\|\omega\|_{W^{1,p}}\leq C (\|\omega\|_{L^p}+\|d\omega\|_{L^p}+\|\ast d\ast\omega\|_{L^p}),
\end{equation}
where the constant $C$ depends on the domain and the exponent $p$ only.

We refer the reader to the paper \cite{IwaniecScottStroffolini} and the book \cite{IwaniecMartin} for more details on Sobolev differential forms.

Let us now see how this applies to the components of Uhlenbeck's decomposition.

The orthogonal mapping $P$ is a function (a $0$-form), and thus its differential $dP$ captures all the derivatives of its coefficients, i.e. $\|\nabla P\|_{L^p}=\|dP\|_{L^p}$. Also, we assume that $P$ restricted to the boundary of the ball $\bb^n$ equals (in the sense of traces) to the identity matrix $\id$, therefore $dP=d(P-\id)$ has vanishing tangent component and Gaffney's inequality \eqref{gaffney} or, since $P$ is a function, standard elliptic estimates give us a comparison between $\|dP\|_{W^{1,p}}$ and $\|dP\|_{L^p}+\|\Delta P\|_{L^p}$.

As for the $(n-2)$-form $\xi$, we have strong assumptions given in \eqref{main}: $\xi$~vanishes on the boundary and it is co-closed. Thus both $\xi$ and $d\xi$ have vanishing tangent components and again Gaffney's inequality \eqref{gaffney} allows us to compare $\|d\xi\|_{W^{1,p}}$ and $\|d\xi\|_{L^p}+\|\Delta \xi\|_{L^p}$, with $\Delta \xi=\ast d\ast\! d\xi$.


\section{Morrey spaces}\label{sec:Morrey}
In this section we recall the properties of Morrey spaces needed in our paper; for more information and for proofs of elementary properties of Morrey spaces we refer the reader e.g. to the monograph \cite{Adams}.

Throughout the paper we customarily use barred integral to denote the integral average and we write $f_E$ for the integral average of $f$ over the set $E$:
$$
f_E=\mvint_E f=\frac{1}{|E|}\int_E f.
$$

Let $U \subset \R^n$ be an open and bounded set. Recall that the Morrey space $L^{p,s}(U)$ is a collection of all functions \mbox{$f\in L^p(U)$} such that
$$
|| f ||_{L^{p, s}}^p = \sup_{ x_0 \in U, r>0}
\left( \frac{1}{r^s} \int_{B_r(x_0) \cap U} | f(y) |^p dy \right) < \infty.
$$

When $s = 0$, the Morrey space $L^{p,0}(U)$ is the same as the usual Lebesgue $L^p(U)$ space. When $s = n$, the dimension of the ambient space, it easily follows from the Lebesgue Differentiation Theorem that the Morrey space $L^{p,n}(U)$ is equivalent to $L^\infty(U)$. Morrey spaces are Banach spaces.

For $1 \leq p \leq q < \infty$  and $ s, \sigma \geq 0$ such that
$ \frac{s -n}{p} \leq \frac {\sigma -n}{q}$ we have
$$
L^{q,\sigma}(U)  \hookrightarrow L^{p,s}(U),
$$
in particular
$$
L^{q,n-q}(U)  \hookrightarrow L^{p,n-p}(U).
$$

\begin{definition}
We define the Morrey-Sobolev space $L^{p,n-kp}_k(U)$ as
$$
L^{p,n-kp}_k(U) = \left\{f \in W^{k,p}(U) : \nabla^j f \in L^{p,n-jp}(U) \quad \text{for $j=1,\ldots,k$} \right\}
$$
with the norm
$$
|| f ||_{L^{p,n-kp}_k(U)} = || f ||_{L^{p}(U)} + \sum_{j=1}^k \| \nabla^j f \|_{L^{p,n-jp}(U)}.
$$
\end{definition}
Note that with this definition, $f\in L_k^{p,n-kp}(U)$ does not imply Morrey estimates for $f$ itself.
\medskip

\begin{definition}

The space $BMO(U)$ consists of these $f\in L^1_{loc}(U)$ for which the expression
$$
[f]_{BMO(U)} =  \sup_{B\subset U}
\left( \mvint_{B} | f(y) - f_B| \right)
$$
is finite.
\end{definition}

Poincar\'e-Wirtinger's inequality on a ball $B$ of radius $r$
$$
\int_{B} | f - f_B |^p \leq C(n,p)r^p \int_B | \nabla f |^p
$$
and H\"older's inequality immediately yield
\begin{equation}\label{bmo morrey comp}
[f]_{BMO(U)} \leq C \| \nabla f \|_{L^{p.n-p}(U)}.
\end{equation}

Thus $L_k^{p,n-kp}(U)\subset BMO(U)$.

\medskip

We say that a domain $U$ is of type $(A)$ if there exists a constant $C>0$ such that for any $x_0 \in U$ and $0<r< diam (U)$
$$
|B_r(x_0) \cap U| \geq C r^n.
$$
This excludes domains with outward cusps.
For domains of type (A) we have the following generalization of Sobolev’s embedding theorem (see~\cite{Morrey}).

\begin{proposition}[Morrey-Sobolev embedding] \label{morrey-sobolev}
Let $U \subset \R^n$ be of type $(A)$; assume that $1\leq p < \infty$ and $\alpha \in (0,p)$. If $f \in W^{1,p}(U)$ is such that $\nabla f \in L^{p,n-p+\alpha}(U)$, then $f \in C^{0,\alpha /p}(U)$.
\end{proposition}

Let us recall the BMO-Gagliardo-Nirenberg type result due to Adams and Frasier \cite{AdamsFrasier} (see also \cite{StrzeleckiGN}):
\begin{proposition}
For any $s>1$ and $f\in W^{2,s}\cap BMO(\R^n)$ there holds
\begin{equation}\label{GN BMO}
\|df\|^2_{L^{2s}(\R^n)}\leq C [f]_{BMO(\R^n)}\|D^2 f\|_{L^s(\R^n)}
\end{equation}
\end{proposition}
This result can be easily localized to functions in $W^{2,s}\cap BMO(B)$, for an arbitrary ball $B\subset\R^n$ (see \cite[Proposition 4.3]{wang-biharm}).
\begin{proposition}
For any ball $B\subset\R^n$ of radius $r$, if $f\in W^{2,p}\cap BMO(B)$, then
\begin{equation}\label{GN BMO local}
\|df\|^2_{L^{2p}(B)}\leq C[f]_{BMO(B)}\left( \|D^2f\|_{L^p(B)}+r^{-1}\|df\|_{L^p(B)}\right)
\end{equation}
\end{proposition}
The proof goes exactly as in \cite{wang-biharm} (where $p$ equals 2): we extend $f$ to a function $\tilde{f}\in W^{2,p}\cap BMO(\R^n)$ and apply the estimate \eqref{GN BMO} to $\tilde{f}$.

As the consequence of the above estimates we obtain the following product estimate.
\begin{lemma}\label{Morrey product}
Assume $B\subset\R^n$ is a ball, $\alpha \in [0,p]$, $u\in L_2^{p,n-2p+\alpha}(B)$ and $v \in L_2^{p,n-2p}(B)$. Then
$$
\|\,|du|\cdot|dv|\,\|_{L^{p,n-2p+\alpha}(B)}\leq C(n,p) \|u\|_{L_2^{p,n-2p+\alpha}(B)}\|dv\|_{L^{2p,n-2p}(B)}.
$$
\end{lemma}
\begin{proof}
First, let us observe that
\begin{equation}\label{trading alpha}
\sup_{B_r\subset B}r^{-\alpha/p} [u]_{BMO(B_r)}\leq C \|du\|_{L^{p,n-p+\alpha}(B)},
\end{equation}
where the supremum is taken over all balls $B_r\subset B$ and the constant depends only on $n$ and $p$.
Indeed, for any such ball $B_r\subset B$ we have
\begin{equation*}
\begin{split}
r^{-\alpha}[u]^{p}_{BMO(B_r)}&=r^{-\alpha}\left(\sup_{B_\rho\subset B_r}\mvint_{B_\rho} |u-u_{B_\rho}|\right)^{p}
\leq C r^{-\alpha}\sup_{B_\rho\subset B_r}\rho^{p-n}\int_{B_\rho}|du|^{p}\\
&\leq C \sup_{B_\rho\subset B_r}\rho^{p-n-\alpha}\int_{B_\rho}|du|^{p}\leq C\sup_{B_\rho\subset B}\rho^{p-n-\alpha}\int_{B_\rho}|du|^{p}\\
&=C\|du\|^{p}_{L^{p,n-p+\alpha}(B)},
\end{split}
\end{equation*}
and the constant $C$ comes from the Poincar\'e  inequality (c.f. \eqref{bmo morrey comp}). Taking supremum over all balls $B_r\subset B$ yields \eqref{trading alpha}.
Now,
\begin{equation*}
\begin{split}
\|\,&|du|\cdot|dv|\,\|^{2p}_{L^{p,n-2p+\alpha}(B)}=\sup_{B_r\subset B} \left( \frac{1}{r^{n-2p+\alpha}}\int_{B_r} |du|^p\cdot |dv|^p \right)^2\\
&\leq \sup_{B_r\subset B} \left(\frac{1}{r^{2n-4p+2\alpha}} \int_{B_r} |du|^{2p} \int_{B_r} |dv|^{2p}\right)\\
&\stackrel{\eqref{GN BMO local}}{\leq} C(n,p)\sup_{B_r\subset B}\left(\frac{1}{r^{2n-4p+2\alpha}}[u]^p_{BMO(B_r)}\left(\int_{B_r}
|D^2 u|^p+\frac{1}{r^{p}}\int_{B_r}|du|^p\right)\right.\\
&\qquad\qquad\qquad\qquad\qquad\times\left.\int_{B_r}|dv|^{2p}\right)\\
&\leq C(n,p)\left(\sup_{B_r\subset B} \frac{1}{r^{\alpha/p}}[u]_{BMO(B_r)}\right)^p\\
&\qquad\qquad\qquad\times \sup_{B_r\subset B} \left(\frac{1}{r^{n-2p+\alpha}}\int_{B_r}|D^2 u|^p+\frac{1}{r^{n-p+\alpha}}\int_{B_r}|du|^p\right)\\
&\qquad\qquad\qquad\times\left. \sup_{B_r\subset B} \frac{1}{r^{n-2p}}\int_{B_r}|dv|^{2p}\right)\\
&\stackrel{\eqref{bmo morrey comp},\eqref{trading alpha}}{\leq}C(n,p) \|u\|^{2p}_{L_2^{p,n-2p+\alpha}(B)}\|dv\|^{2p}_{L^{2p,n-2p}(B)}.
\end{split}
\end{equation*}

\end{proof}

In particular, for $\alpha \in [0,p]$ we obtain (see also \cite[Proposition 3.2]{Struwe})
\begin{align} \label{kwadrat w Morreyach}
\| du \|^2_{L^{2p,n-2p+\alpha}(B)}
& \leq
C \| du \|_{L^{2p,n-2p}(B)} \| u \|_{L_2^{p,n-2p+\alpha}(B)} \\
& \leq
C \| u \|^2_{L_2^{p,n-2p+\alpha}(B)}. \nonumber
\end{align}


\section{Uhlenbeck's decomposition, the case $n/2 \leq p < n$}\label{sec:Sobolev}

In this section we prove Theorem~\ref{thm main}. We first state it (as Lemma~\ref{thm main-b}) and prove it in the case $p \in (n/2,n)$; the case $p = n/2$ follows by approximation (see Corollary \ref{corr} at the end of this section).

\begin{lemma}\label{thm main-b}
Let $\frac{n}{2}<p<n$. There exists $\epsilon >0$ such that for any antisymmetric matrix $\Omega \in W^{1,p}(\bb^n, so(m)\otimes \Lambda^1\R^n)$  of $1$-differential forms on $\bb^n$ such that
$$
\|\Omega \|_{L^n}<\epsilon
$$
there exist $P\in W^{2,p}(\bb^n,SO(m))$, and $\xi \in W^{2,p}\cap W_0^{1,p}(\bb^n, so(m)\otimes \Lambda^{n-2}\R^n)$ satisfying the system
\begin{equation}\label{main}
\left\{
\begin{aligned}
P^{-1}dP+P^{-1}\Omega P&=\ast d\xi &&\text{on }\bb^n,\\
d\ast\xi&=0 &&\text{on }\bb^n,\\
\xi&=0 &&\text{on } \partial \bb^n, \\
P&=\mathrm{Id} &&\text{on } \partial \bb^n;
\end{aligned}
\right.
\end{equation}
and such that
\begin{subequations}\label{thm main:est}
\begin{align}
\|d\xi \|_{L^p}+\|dP\|_{L^p}&\leq C(n,m,p)\|\Omega \|_{L^p},\label{thm main:est2}\\
\|d\xi \|_{L^n}+\|dP\|_{L^n}&\leq C(n,m)\|\Omega \|_{L^n},\label{thm main:est3}\\
\|d\xi\|_{W^{1,p}}+\|dP\|_{W^{1,p}}&\leq C(n,m,p)\|\Omega \|_{W^{1,p}}.\label{thm main:est1}
\end{align}
\end{subequations}
\end{lemma}

\begin{remark}
In what follows, we write, to keep the notation simple, \mbox{$\Omega \in W^{1,p}$} ($P,\xi \in W^{2,p}$ etc.), instead of  $\Omega \in W^{1,p}(\bb^n, so(m)\otimes \Lambda^1\R^n)$.
\end{remark}

\bigskip

We shall break the proof of Lemma~\ref{thm main-b} into several lemmata. Following Rivi\`ere, we introduce sets
\begin{align*}
V_\epsilon &=\{\Omega \in W^{1,p}&~:~&\|\Omega \|_{L^n}<\epsilon \},\\
U_{\epsilon}&=\{\Omega \in W^{1,p}&~:~&\|\Omega \|_{L^n}<\epsilon \text{ and there exist $P$ and $\xi$ }\\
&&&\text{ satisfying the system \eqref{main} and the estimate \eqref{thm main:est}}\}.
\end{align*}

We show that for $\epsilon>0 $ and sufficiently small the set $U_\epsilon $ is closed and open in $V_\epsilon $, and since the latter is path connected (it is star-shaped in $W^{1,p}$), it follows that $U_\epsilon =V_\epsilon $, which proves Lemma~\ref{thm main-b}.

\begin{lemma}\label{lemma:closedness}
The set $U_\epsilon$ is closed in $V_\epsilon$.
\end{lemma}

\begin{proof}
Suppose $(\Omega_k)$ is a sequence in $U_\epsilon $, convergent in $W^{1,p}$ to some $\Omega$.
With every $\Omega_k$ we associate $P_k$ and $\xi_k$ that satisfy \eqref{main}:
\begin{equation}\label{thm main:system}
\left\{
\begin{aligned}
P_k^{-1}dP_k+P_k^{-1}\Omega_k P_k&=\ast d\xi_k &&\text{on }\bb^n,\\
d\ast\xi_k&=0 &&\text{on }\bb^n,\\
\xi_k&=0&&\text{on } \partial \bb^n,\\
P_k&=\mathrm{Id} &&\text{on } \partial \bb^n;
\end{aligned}
\right.
\end{equation}
and the estimates \eqref{thm main:est} hold for $P_k$, $\xi_k$ and $\Omega_k$,, in particular
\begin{equation}\label{boundedsequence}
\|d\xi_k\|_{W^{1,p}}+\|dP_k\|_{W^{1,p}}\leq C(n,m) \|\Omega_k\|_{W^{1,p}}.
\end{equation}
The boundary condition on $\xi_k$ and boundedness of $P_k$ (recall that ${|P_k|=1}$) allow us to interpret \eqref{boundedsequence} as boundedness of $P_k$ and $\xi_k$ in $W^{2,p}$, since the sequence $(\Omega_k)$, being convergent, is necessarily bounded in $W^{1,p}$. We can thus assume (after passing to subsequences) that $P_k$ and $\xi_k$ are weakly convergent in $W^{2,p}$ to some $P$ and $\xi$.

Both the boundary condition $\xi|_{\partial\bb}=0$ and the condition $d\ast\xi=0$ are preserved when passing to the weak limit. Moreover, possibly after passing to a subsequence, we have
\begin{align*}
\Omega_k &\rightarrow\Omega &\text{ in }W^{1,p}&&~~\Rightarrow~~&&\Omega_k &\rightarrow\Omega &\text{ in }L^{\frac{np}{n-p}},\\
P_k &\rightharpoonup P &\text{ in }W^{2,p}&&~~\Rightarrow~~&&dP_k &\rightharpoonup dP &\text{ in }L^{\frac{np}{n-p}},\\
\xi_k &\rightharpoonup \xi &\text{ in }W^{2,p}&&~~\Rightarrow~~&&d\xi_k &\rightharpoonup d\xi &\text{ in }L^{\frac{np}{n-p}},
\end{align*}
and for any small $\delta >0$, $\Omega_k$, $dP_k$ and $d\xi_k$  converge strongly (to $\Omega$, $P$ and $\xi$, respectively) in $L^s$, with $s=\frac{np}{n-p}-\delta$, in particular in $L^n$, since $np/(n-p) > n$. Also, $P_k$ are uniformly bounded in $L^{\infty}$ and strongly convergent, by Sobolev embedding theorem, in $L^{q}$ for any $q$. This is enough to show the strong convergence of $\Omega_k P_k$ to $\Omega P$ in $L^s$; altogether, we may pass to the strong limit in $L^s$ in the system \eqref{thm main:system}, showing that the equation
$$
dP+\Omega P=P\ast d\xi \qquad \text{in }\bb^n
$$
is satisfied in the sense of distributions.

The estimates \eqref{thm main:est} for $P$ and $\xi$ are obvious.\\
\end{proof}

\begin{remarknn}
For any $P$ and $\xi$ in $W^{1,p}$ that satisfy \eqref{thm main:system} we have that $d\xi\in W^{1,p}$ implies $P\in W^{2,p}$. Indeed, we have $dP=P\ast d\xi-\Omega P$, and for $p>n/2$ the right hand side is in $W^{1,p}$.
\end{remarknn}

\medskip

Now we proceed to prove the openness of $U_\epsilon$. In contrast with the previous lemma this is more delicate; we split the reasoning again into several lemmata.

\begin{lemma}\label{lemma:existence}
There exists a constant $\kappa =\kappa (n)>0$ such that for any $\zeta\in W^{2,p}(\bb^n,so(m)\otimes \Lambda^{n-2}\R^n)$ with $\|d\zeta\|_{L^n}\leq \kappa $ there exists $\eta >0$ such that for any $\lambda \in W^{1,p}(\bb^n,so(m)\otimes\Lambda^1\R^n)$ with $\|\lambda \|\leq \eta$ the equation
\begin{equation}\label{implicit}
\ast d \ast\left(Q^{-1}dQ+Q^{-1}(\ast d\zeta+\lambda)Q\right)=0
\end{equation}
has a solution $Q=Q(\lambda)\in W^{2,p}(\bb^n,SO(m))$ such that $Q = Id$ on $\partial \bb^n$. Moreover, $Q(\lambda)$ depends continuously on $\lambda$.
\end{lemma}
Note that \eqref{implicit} implies, through Poincar\'e's lemma, that the term in parentheses is of the form $\ast d\tilde{\zeta}$, for some antisymmetric $(n-2)$-form $\tilde{\zeta}$. The above Lemma should be understood as follows: Uhlenbeck's decomposition (i.e. $Q$ and $\tilde{\zeta}$) exists if our matrix is a small (in $W^{1,p}$) perturbation of a co-closed form $\ast d\zeta$, provided $\| d\zeta \|_{L^n}$ is sufficiently small.

\begin{proof}
Since we are interested in finding any $Q\in W^{2,p}(\bb^n,SO(m))$ satisfying \eqref{implicit} and the boundary condition, we shall look for one of the form $e^u$, where $u\in \bs =W^{2,p}(\bb^n, so(m))\cap W_o^{1,p}(\bb^n, so(m))$.
We define the operator $$T:\bs\times W^{1,p}(\bb^n, so(m)\otimes\Lambda^1\R^n)\to L^p(\bb^n)$$ by
\begin{equation*}
T(u,\lambda )=\ast d\ast (e^{-u} de^u+e^{-u}(\ast d\zeta+\lambda)e^u).
\end{equation*}
This is a well defined, smooth operator. Using Implicit Function Theorem, we prove that for any sufficiently small $\lambda$ the equation $T(u,\lambda )=0$  has a~solution $u_\lambda $, continuously depending on $\lambda $. To this end we linearize $T$ at $(u,\lambda )=(0,0)$ with respect to the first argument:
\begin{equation*}\label{T linearized}
H(\psi)=\ast d \ast (d\psi+[\ast d\zeta,\psi])=\Delta \psi+\ast[d\zeta,d\psi], \qquad H\colon \bs \to L^p(\bb^n),
\end{equation*}
where the commutator $[\cdot,\cdot]$ denotes a commutator of two $so(m)$ matrices.

We have, by H\"older's inequality
\begin{equation*}
\begin{split}
\|\Delta \psi\|_{L^p}&\leq \|H(\psi)\|_{L^p}+\|[d\zeta,d\psi]\|_{L^p}\\
&\leq \|H(\psi)\|_{L^p}+2\|d\zeta\|_{L^n}\| \|d\psi\|_{L^{np/(n-p)}},
\end{split}
\end{equation*}
thus
\begin{equation*}
\begin{split}
\|H(\psi)\|_{L^p}
&\geq \|\Delta \psi\|_{L^p}-2\|d\zeta\|_{L^n}\| \|d\psi\|_{L^{np/(n-p)}}\\
&\geq \|\Delta \psi\|_{L^p}-2C_S\|d\zeta\|_{L^n}\| \|d\psi\|_{W^{1,p}}\\
&\geq C(\|d\psi\|_{W^{1,p}}-2\kappa \|d\psi\|_{W^{1,p}})\\
&=C(1-2\kappa)\|d\psi\|_{W^{1,p}},
\end{split}
\end{equation*}
where the constant $C_S$ comes from the Sobolev embedding (recall that ${\psi\in \bs}$ is a matrix-valued \emph{function}, i.e. a $0$-form, vanishing at the boundary). Therefore, for $\kappa$ small, $H$ is injective.

Showing surjectivity of $H$ amounts to showing that the system
\begin{equation}\label{H surjective}
H(\psi)=\Delta \psi+\ast[d\zeta,d\psi]=f
\end{equation}
has a solution in $\bs$ for arbitrary $f\in L^p$.

Let us consider an operator $K:\bs\to\bs$, with $K(\psi)$ defined as a solution to the system
\begin{equation}\label{system for K}
\begin{cases}\Delta K(\psi)=-\ast[d\zeta,d\psi]&\text{ in }\bb^n,\\
K(\psi)=0&\text{ on }\partial\bb^n.
\end{cases}
\end{equation}
Using H\"older's and Sobolev's inequalities and the fact that the Newtonian potential $\Delta^{-1}:L^p\to\bs$ is continuous, we get
\begin{equation*}
\begin{split}
\|K(\psi)\|_{\bs}&\leq \vertiii{\Delta^{-1}}_{L^p\to \bs}\, \|\ast [d\zeta,d\psi]\|_{L^p}\\
&\leq 2\vertiii{\Delta^{-1}}_{L^p\to \bs}\,\|d\zeta\|_{L^n}\|d\psi\|_{L^{np/(n-p)}}\\
&\leq 2\vertiii{\Delta^{-1}}\kappa \|\psi\|_{\bs}.
\end{split}
\end{equation*}
For $\kappa$ sufficiently small we can have $\vertiii{K}_{\bs\to\bs}$ small and $Id-K:\bs\to\bs$ invertible.

Let now $\phi=(Id-K)\psi$. If $\psi$ is a solution to \eqref{H surjective}, then
$$
\Delta \phi=\Delta\psi-\Delta K(\psi)=\Delta \psi+\ast[d\zeta,d\psi]=f,
$$
and we can solve \eqref{H surjective} for any $f\in L^p$ by solving the above Poisson equation and applying to its solution the inverse mapping to $Id-K$.

Altogether, $H:\bs \to L^p$ is an isomorphism, and we can apply the Implicit Function Theorem to get $u_\lambda$ as a continuous function of $\lambda$.

To end the proof of the lemma, we take $Q(\lambda)=e^{u_\lambda}$.
\end{proof}

\begin{lemma}\label{a priori}
Suppose $n/2<p<n$. There exists $\kappa=\kappa(p,n)$ such that for any $\Omega\in V_\epsilon $ and $P$, $\xi $ in $W^{2,p}$ satisfying the system \eqref{main} and additionally the estimate
\begin{equation}\label{kappa estimate}
\|dP\|_{L^n}+\|d\xi \|_{L^n}<\kappa
\end{equation}
the estimates \eqref{thm main:est} hold.
\end{lemma}
\begin{proof}
The lemma follows from rather standard elliptic estimates, but we include them here for the sake of completeness.

We have
\begin{equation}\label{laplasjan xi}
\begin{split}
\Delta \xi&=(\ast d\ast d+d\ast d\ast )\xi=\ast d(\ast d\xi)\\
&=\ast d(P^{-1}dP+P^{-1}\Omega P)\\
&=\ast(dP^{-1}\wedge dP)+\ast d(P^{-1}\Omega P).
\end{split}
\end{equation}
Note that for $q=p/(p-1)$, $\|d\xi\|_{L^p}$ is equivalent to
\begin{equation*}
\sup_{\|d \phi\|_{L^q}\leq 1}\int_{\bb^n} d\xi \cdot d\phi,
\end{equation*}
where $\phi$ is a smooth, compactly supported (in particular with null boundary values)  $(n-2)$-form on $\bb^n$. The inequality
$$
\|d\xi\|_{L^p}
=\sup_{\|\eta\|_{L^q}\leq 1}\int_{\bb^n} d\xi \cdot \eta
\geq \sup_{\|d \phi\|_{L^q}\leq 1}\int_{\bb^n} d\xi \cdot d\phi
$$
is obvious. Applying the Hodge decomposition to the $(n-1)$-form $\eta$, $\eta=d\phi+\psi$ with $\delta \psi=0$, $\|d\phi\|_{L^q}\leq C_q \|\eta\|_{L^q}$, we get
\begin{equation*}\begin{split}
\|d\xi\|_{L^p}&=\sup_{\|\eta\|_{L^q}\leq 1}\int_{\bb^n} d\xi \cdot \eta\leq \sup_{\|\eta\|_{L^q}\leq 1}\left(\int_{\bb^n} d\xi \cdot d\phi + \int_{\bb^n} d\xi \cdot \psi \right)\\
&=C_q \sup_{\|\eta\|_{L^q}\leq 1}\left(\int_{\bb^n} d\xi \cdot d\frac{\phi}{C_q} - \int_{\bb^n} \xi \wedge \delta\psi \right)
\leq C_q \sup_{\|d \tilde{\phi}\|_{L^q}\leq 1}\int_{\bb^n} d\xi \cdot d\tilde{\phi}\,,
\end{split}
\end{equation*}
where $\tilde{\phi}=\phi/C_q$.

Denote by $\bar{P}$ the mean value of $P$ over $\bb^n$: $\bar{P}=\mvint_{\bb^n}P$. For any $\phi$ as above, with $\|d\phi\|_{L^q}\leq 1$,
\begin{equation*}
\begin{split}
\int_{\bb^n}d\xi\cdot d\phi&=-\int \Delta \xi\cdot \phi\\
&=-\int_{\bb^n} \big[ \ast \left(dP^{-1}\wedge d(P-\bar{P})\right) + \ast d \left( P^{-1}\Omega P\right)\big] \cdot \phi\\
&=-\int_{\bb^n} \big[ dP^{-1}(P-\bar{P}) +P^{-1}\Omega P \big] \wedge d\phi \\
&\leq C\|dP^{-1}\|_{L^p} \|P-\bar{P}\|_{BMO} \|d\phi\|_{L^q} +\|P^{-1}\Omega P\|_{L^p}\|d\phi\|_{L^q}\\
&\leq C\|dP\|_{L^p}\|dP\|_{L^n}+\|\Omega \|_{L^p}\\
&\leq \kappa C\|dP\|_{L^p}+\|\Omega \|_{L^p},
\end{split}
\end{equation*}
thus
\begin{equation}\label{a1}
\| d \xi \|_{L^p} \leq C_q (  \kappa C\|dP\|_{L^p}+\|\Omega \|_{L^p} ),
\end{equation}
with the constants $C$ (possibly different in every line) dependent only on $n$ and $p$. Note that, since $P$ is an orthogonal matrix, $|P|=|P^{-1}|=1$ and $|dP^{-1}|=|dP|$. In the estimate above we use, for the first summand, the Coifman-Lions-Meyer-Semmes div-curl inequality (\cite{CLMS}) and later the standard inclusion $W^{1,n}\hookrightarrow BMO$; the second summand is estimated by H\"older's inequality.

On the other hand, taking $L^p$ norms of both sides of the equation
\begin{equation}\label{eq for dP}
dP=P\ast d\xi+ \Omega P
\end{equation}
(c.f \eqref{main}) gives
\begin{equation}\label{a2}
\|dP\|_{L^p}\leq \|d\xi \|_{L^p}+\|\Omega \|_{L^p}.
\end{equation}
Putting \eqref{a1} and \eqref{a2} together we get
\begin{equation*}
\begin{split}
\|d\xi\|_{L^p}+\|dP\|_{L^p}&\leq 2\|d\xi\|_{L^p}+\|\Omega \|_{L^p}\\
&\leq C_1(n,p)\kappa \|dP\|_{L^p}+C_2(n,p)\|\Omega \|_{L^p},
\end{split}
\end{equation*}
and for $\kappa <\frac{1}{C_1}$ this implies that the estimate \eqref{thm main:est2} holds.

The above calculation is valid also for $p=n$, which yields the estimate \eqref{thm main:est3}.

To show the estimate \eqref{thm main:est1}, by taking $\ast d \ast$ of both sides of \eqref{eq for dP}, we see that
\begin{equation*}
\begin{split}
\Delta P&=\ast d\ast d P=\ast d\ast (P\ast d\xi+ \Omega P)\\
&=\ast(dP\wedge d\xi)+\ast d \ast (\Omega P),
\end{split}
\end{equation*}
thus
\begin{equation}\label{dP in sobolev}
\begin{split}
\|dP\|_{W^{1,p}}&\leq C \big( \|dP\|_{L^p}+ \|\Delta P\|_{L^p}\big) \\
&\leq C \big( \|dP\|_{L^p}+\|dP\wedge d\xi\|_{L^p}+\|d\Omega P\|_{L^p}+\|\Omega \wedge dP\|_{L^p} \big)\\
&\leq C \big( \|dP\|_{L^p}+\|dP\|_{L^n}\|d\xi\|_{L^{np/(n-p)}}+\|d\Omega \|_{L^p}+\|\Omega\|_{L^{np/(n-p)}}\|dP\|_{L^n} \big) \\
&\leq C\big( \|dP\|_{L^p}+\kappa \|d\xi\|_{W^{1,p}}+(1+\kappa)\|\Omega \|_{W^{1,p}} \big).
\end{split}
\end{equation}
The constant $C$ comes from Gaffney's inequality \eqref{gaffney}, standard elliptic estimates and the Sobolev embedding $W^{1,p}\hookrightarrow L^{np/(n-p)}$, thus it depends only on $p$ and $n$.

Similarly, using \eqref{laplasjan xi},
 \begin{equation}\label{dxi in sobolev}
\begin{split}
\|d\xi\|_{W^{1,p}}&\leq C \big( \|d\xi\|_{L^p}+\|\Delta \xi\|_{L^p}\big) \\
&= C\big( \|d\xi\|_{L^p}+\| \ast ( dP^{-1}\wedge dP)+\ast d(P^{-1}\Omega P)\|_{L^p} \big)\\
&\leq C\big( \|d\xi\|_{L^p}+\|dP\|_{L^n}\|dP\|_{L^{np/(n-p)}}+\|dP\|_{L^n}\|\Omega P\|_{L^{np/(n-p)}} \\
&\qquad +\|P^{-1}d\Omega P\|_{L^p}+\|P^{-1}\Omega \|_{L^{np/(n-p)}}\|dP\|_{L^n} \big)\\
&\leq C \big( \|d\xi\|_{L^p}+\kappa \|dP\|_{W^{1,p}}+(\kappa+1)\|\Omega\|_{W^{1,p}} \big).
\end{split}
\end{equation}
Note that, as pointed out in Section \ref{sec:forms}, the full Sobolev norms of $dP$ and $d\xi$ can be estimated with the norms of Laplacians of $P$ and $\xi$, thanks to the boundary conditions they satisfy.

Composing \eqref{dP in sobolev} and \eqref{dxi in sobolev} with the already proved estimate \eqref{thm main:est2} we get
$$
\|dP\|_{W^{1,p}}+\|d\xi\|_{W^{1,p}}
\leq C\big(
\kappa \|dP\|_{W^{1,p}}+\kappa  \|d\xi\|_{W^{1,p}}+ (\kappa +1)\|\Omega\|_{W^{1,p}} \big),
$$
which, for $\kappa$ sufficiently small, yields the estimate \eqref{thm main:est1}.
\end{proof}

\begin{lemma}\label{lemma:openness}
The set $U_\epsilon$ is, for $\epsilon $ sufficiently small, open in $V_\epsilon $.
\end{lemma}

\begin{proof}

Choose $\Omega_o\in U_\epsilon$ and let $P_o$ and $\xi_o$ be the orthogonal transformation and antisymmetric $(n-2)$-form associated with $\Omega_o$, so that Theorem \ref{thm main} holds for $\Omega_o$, $P_o$ and $\xi_o$.

Take now $\Omega\in V_\epsilon$ close to $\Omega_o$ in $W^{1,p}$: we ask that for $\lambda =P_o^{-1}(\Omega -\Omega_o)P_o$ we have $\|\lambda \|_{W^{1,p}}<\eta $ (the conjugation with $P_o\in W^{2,p}$ is continuous in $W^{1,p}$). Applying Lemma \ref{lemma:existence} with $\zeta =\xi_o$, we find $Q\in W^{2,p}(\bb^n,SO(m))$ such that
\begin{equation}\label{open:1}
\ast d \ast\Big(Q^{-1}dQ+Q^{-1}(\ast d\xi_o+P_o^{-1}(\Omega -\Omega_o)P_o)Q\Big)=0.
\end{equation}
Setting $P=P_o Q$ we see that \eqref{open:1} reduces to
$$
\ast d \ast (P^{-1}dP+P^{-1}\Omega  P)=0.
$$
By Poincar\'e's Lemma, this implies that $P^{-1}dP+P^{-1}\Omega  P$  is a coexact form, i.e. there exists an antisymmetric  $(n-2)$-form $\xi$ such that
\begin{equation}\label{last}
\ast d\xi=P^{-1}dP+P^{-1}\Omega  P,
\end{equation}
thus $P$ and $\xi$ give Uhlenbeck's decomposition of $\Omega$.

Note that $Q$ and $P_o \in W^{2,p}\cap L^\infty$ imply that $P\in W^{2,p}$.
By the Hodge decomposition theorem we can choose $\xi$ to be coclosed ($d\ast\xi=0$ on $\bb^n$) and to have zero boundary values ($\xi|_{\partial \bb^n}=0$). Finally, the right hand side of \eqref{last} is in $W^{1,p}$, which gives $\xi\in W^{2,p}$.

What remains to prove is that $P$, $\xi$ and $\Omega $ satisfy the estimates \eqref{thm main:est}. Observe that if $\|\Omega -\Omega_o\|_{W^{1,p}}$ is small enough, then by continuity of the mapping $\lambda \mapsto u_\lambda$ so is  $\|P-P_o\|_{W^{1,p}}$ and $\|\xi-\xi_o\|_{W^{1,p}}$; choosing $\eta$ (which measures the distance $\|\Omega -\Omega_o\|_{W^{1,p}}$) sufficiently small we may have
$$
\|P-P_o\|_{W^{1,p}}+\|\xi-\xi_o\|_{W^{1,p}}<\epsilon.
$$
We also know that
$$
\|d\xi_o\|_{L^n}+\|dP_o\|_{L^n}\leq C\|\Omega_o\|_{L^n}\leq C\epsilon
$$
(this follows from $\Omega_o\in U_\epsilon$).

Taking $\epsilon$ sufficiently small, we may ensure that
\begin{equation*}
\begin{split}
\|d\xi\|_{L^n}+\|dP\|_{L^n}&\leq\|d\xi-d\xi_o\|_{L^n}+\|P-P_o\|_{L^n}+\|d\xi_o\|_{L^n}+\|dP_o\|_{L^n}\\
&<(C+1)\epsilon<\kappa,
\end{split}
\end{equation*}
with $\kappa$ as in Lemma \ref{a priori}. Applying this lemma we show that the estimates \eqref{thm main:est} hold.

Altogether, $\Omega\in U_\epsilon$, which proves the openness of $U_\epsilon$.

\end{proof}

\begin{proof}[Proof of Lemma \ref{thm main-b}]
Since, by Lemmata \ref{lemma:closedness} and \ref{lemma:openness}, for $\epsilon $ sufficiently small the set $U_\epsilon $ is closed and open in $V_\epsilon $, and since the latter is path connected (it is star-shaped in $W^{1,p}$), it follows that $U_\epsilon =V_\epsilon $, which proves Lemma \ref{thm main-b}.

\end{proof}

\medskip

It is worth noting that, for $\Omega \in W^{1,n/2}$, the proof of existence of the decomposition, i.e. Lemma \ref{lemma:existence}, fails. However, we can proceed by a standard density argument: approximate $\Omega$ in $W^{1, {n/2}}$ with $\Omega_k$ in $W^{1,p}$ for $p > n/2$ and argue as in Lemma \ref{lemma:closedness} (see also the proof of Theorem \ref{u-m}), obtaining
\begin{corollary}\label{corr}
Let $\Omega \in W^{1,n/2}$. There exists $\epsilon >0$ such that if $\Omega$ is an antisymmetric matrix of $1$-differential forms on $\bb^n$ such that
$$
\|\Omega \|_{L^n}<\epsilon
$$
then there exist $P\in W^{2,n/2}(\bb^n,SO(m))$ and $\xi \in W^{2,n/2}(\bb^n, so(m)\otimes \Lambda^{n-2}\R^n)$ satisfying the system
\begin{equation*}\label{main:corr}
\left\{
\begin{aligned}
P^{-1}dP+P^{-1}\Omega P&=\ast d\xi &&\text{on }\bb^n,\\
d\ast\xi&=0 &&\text{on }\bb^n,\\
\xi&=0 &&\text{on } \partial \bb^n, \\
P&=\mathrm{Id} &&\text{on } \partial \bb^n;
\end{aligned}
\right.
\end{equation*}
and such that
\begin{subequations}\label{corr:est}
\begin{align*}
&\|d\xi \|_{W^{1,n/2}}+\|dP\|_{W^{1,n/2}}\leq C(n,m)\|\Omega \|_{W^{1,n/2}},\\
&\|d\xi \|_{L^n}+\|dP\|_{L^n}\leq C(n,m)\|\Omega \|_{L^n} < C \epsilon.
\end{align*}
\end{subequations}
\end{corollary}
Lemma~\ref{thm main-b} and Corollary~\ref{corr} together yield Theorem~\ref{thm main}.


\section{Uhlenbeck's decomposition, the case $1< p < n/2$} \label{sec:MorreySobolev}
In this section we prove Theorem~\ref{u-m}.
\begin{theorem*}
Assume $1< p < n/2$. Let
$$
\Omega : \bb^n \to so(m)\otimes \Lambda^1\R^n
$$
be an antisymmetric matrix of $1$-differential forms on $\bb^n$. Assume
$$
\Omega \in L^{2p,n-2p} \quad \text{and} \quad d\Omega \in L^{p,n-2p}
$$
There exists $\epsilon >0$ such that if $\Omega$ satisfies the smallness condition
$$
\|\Omega \|_{L^{2p,n-2p}}<\epsilon
$$
then there exist $P\in W^{2,p}(\bb^n,SO(m))$ and $\xi \in W^{2,p}(\bb^n, so(m)\otimes \Lambda^{n-2}\R^n)$ satisfying the system
\begin{equation}\label{main2}
\left\{
\begin{aligned}
P^{-1}dP+P^{-1}\Omega P&=\ast d\xi &&\text{on }\bb^n,\\
d\ast\xi&=0 &&\text{on }\bb^n,\\
\xi&=0 &&\text{on } \partial \bb^n, \\
P&=\mathrm{Id} &&\text{on } \partial \bb^n.
\end{aligned}
\right.
\end{equation}
Moreover $P, \xi \in L^{p, n-2p}_2$ with
\begin{subequations}\label{u-m:est}
\begin{align}
\|d\xi \|_{L^{p, n-p}}+\|dP\|_{L^{p,n-p}}& \leq  C(n,m)\|\Omega \|_{L^{2p, n-2p}} \label{u-m:est1}\\
\|\Delta \xi \|_{L^{p, n-2p}}+\| \Delta P\|_{L^{p,n-2p}}& \leq
C(n,m)( \|\Omega \|_{L^{2p, n-2p}} + \|d\Omega \|_{L^{p, n-2p}}). \label{u-m:est2}
\end{align}
\end{subequations}
\end{theorem*}

\medskip

\begin{remarknn}
Observe that for $p=n/2$, by the Sobolev Embedding Theorem, we have automatically $\Omega \in L^n$. The Morrey space $L^{2p, n-2p}$ equals $L^n$ in this case and the smallness condition for the norm of $\Omega$ agrees with the one in Theorem~\ref{thm main}.
\end{remarknn}

As in Section~\ref{sec:Sobolev}, we shall break the proof of Theorem \ref{u-m} into several lemmata.
The proof of the existence of $P$ and $\xi$ (Lemma \ref{lemma:existence}) cannot be adapted to the present situation.
To avoid this difficulty, we first prove the theorem under more stringent regularity assumptions (see Lemma \ref{thm main mr} below). To prove the existence of the elements of decomposition we follow the strategy of Tao and Tian \cite{TaoTian}.
At a certain moment of the proof (Lemma \ref{lemma:existence2}) we use the fact that due to the Morrey-Sobolev embedding (Proposition \ref{morrey-sobolev}), for $\alpha > 0$,
$$
L_1^{p,n-p+\alpha} \hookrightarrow C^0.
$$
This is not true for $L_1^{p,n-p}$. Also, as pointed out in \cite{zorko}, continuous functions are not dense in~$L^{p,s}$.

\begin{lemma}\label{thm main mr}
Let $1< p < n/2$. There exists $\epsilon >0$ such that for every $\alpha > 0$ and for every
$$
\Omega \in L^{2p,n-2p+\alpha} \quad \text{such that} \quad d\Omega \in L^{p,n-2p},
$$
if $\Omega$ satisfies the smallness condition
$$
\|\Omega \|_{L^{2p,n-2p}}<\epsilon
$$
then there exist $P, \xi  \in L^{p, n-2p+\alpha}_2$ satisfying the system \eqref{main2}
and the estimates
\begin{subequations}\label{thm main:est mr}
\begin{align}
\|d\xi \|_{L^{p, n-p}}+\|dP\|_{L^{p,n-p}} &\leq  C(n,m)\|\Omega \|_{L^{2p, n-2p}} <C \epsilon \label{thm main:est1 mr},\\
\|d\xi \|_{L^{p, n-p+\alpha}}+\|dP\|_{L^{p,n-p+\alpha}} &\leq  C(n,m)\|\Omega \|_{L^{2p, n-2p+\alpha}} \label{thm main:est2 mr},\\
\|\Delta \xi \|_{L^{p, n-2p}}+\| \Delta P\|_{L^{p,n-2p}} &\leq
C(n,m)\left( \|\Omega \|_{L^{2p, n-2p}}\right.\label{thm main:est3 mr}\\\notag
&\qquad \qquad \quad + \left.\|d \Omega \|_{L^{p, n-2p}}\right).
\end{align}
\end{subequations}
\end{lemma}

\medskip

\begin{proof}[Proof of the Lemma \ref{thm main mr}]

As in the Sobolev case, for $\alpha$, $\epsilon >0$ we introduce sets
\begin{gather*}
V^\alpha_\epsilon =\{\Omega \in L^{2p, n-2p+\alpha} \colon d\Omega \in L^{p,n-2p} \text{ and }\|\Omega \|_{L^{2p, n-2p}}<\epsilon \} \\
U^\alpha_\epsilon =\{\Omega \in L^{2p, n-2p+\alpha} \colon d\Omega \in L^{p,n-2p} \text{ and } \|\Omega \|_{L^{2p,n-2p}}<\epsilon \\
\text{ and there exist $P$ and $\xi$
satisfying the system \eqref{main2} and estimates \eqref{thm main:est mr}}\}
\end{gather*}

In Lemmata \ref{lemma:closedness2} and \ref{lemma:openness2} below we show that for $\epsilon $ sufficiently small the set $U^\alpha_\epsilon $ is closed and open in $V^\alpha_\epsilon $, and since the latter is path connected (it is star-shaped), it follows that $U^\alpha_\epsilon =V^\alpha_\epsilon $. This (up to the proofs of these lemmata) completes the proof of the lemma.\\
\end{proof}

\begin{lemma}\label{lemma:closedness2}
The set $U^\alpha_\epsilon$ is closed in $V^\alpha_\epsilon$.
\end{lemma}

\begin{proof}

Suppose $(\Omega_k)$ is a sequence in $U_\epsilon^\alpha $ convergent in $L_1^{p,n-2p}$ to some $\Omega$.
Observe that $L_1^{p,n-2p}$ embeds continuously in $W^{1,p}$. Therefore the sequence $(\Omega_k)$ is convergent in $W^{1,p}$.

With every $\Omega_k$ we have associated $P_k$, $\xi_k$ that satisfy \eqref{main2}:
\begin{equation*}
\left\{
\begin{aligned}
P_k^{-1}dP_k+P_k^{-1}\Omega_k P_k&=\ast d\xi_k &&\text{on }\bb^n,\\
d\ast\xi_k&=0 &&\text{on }\bb^n,\\
\xi_k&=0&&\text{on } \partial \bb^n.\\
\end{aligned}
\right.
\end{equation*}

We also have the estimates \eqref{thm main:est mr}, in particular
\begin{align*}
\|d\xi_k \|_{L^{p, n-p+\alpha}}+\|dP_k\|_{L^{p,n-p+\alpha}}& \leq  C(n,m)\|\Omega_k \|_{L^{2p, n-2p+\alpha}} \\
\|\Delta \xi_k \|_{L^{p, n-2p}}+\| \Delta P_k\|_{L^{p,n-2p}}& \leq
C(n,m)( \|\Omega_k \|_{L^{2p, n-2p}} + \|d \Omega_k \|_{L^{p, n-p}}).
\end{align*}
The inclusion of $L_1^{p,n-p+\alpha}$ in $W^{1,p}$, the boundary condition on $\xi_k$ and boundedness of $P_k$ ($|P_k|=1$) allow us to interpret the above as boundedness of $P_k$ and $\xi_k$ in $W^{2,p}$. We can thus assume (after passing to subsequences) that $P_k$ and $\xi_k$ are weakly convergent in $W^{2,p}$ to some $P$ and $\xi$.

Then, we argue as in Lemma \ref{lemma:closedness}: Both the boundary condition $\xi|_{\partial\bb}=0$ and the condition $d\ast\xi=0$ are preserved when passing to the weak limit. Moreover, since $n > 2p > p$, after passing to a subsequence,
\begin{align*}
\Omega_k &\rightarrow\Omega &\text{ in }W^{1,p}&&~~\Rightarrow~~&&\Omega_k &\rightarrow\Omega &\text{ in }L^{\frac{np}{n-p}},\\
P_k &\rightharpoonup P &\text{ in }W^{2,p}&&~~\Rightarrow~~&&dP_k &\rightarrow dP &\text{ in }L^{n},\\
\xi_k &\rightharpoonup \xi &\text{ in }W^{2,p}&&~~\Rightarrow~~&&d\xi_k &\rightharpoonup d\xi &\text{ in }L^{\frac{np}{n-p}},
\end{align*}
and for any small $\delta >0$, $\Omega_k$, $dP_k$ and $d\xi_k$  converge strongly (to $\Omega$,$P$ and $\xi$, respectively) in $L^s$, for any $s<\frac{np}{n-p}$. We also know that $P_k$ are uniformly bounded in $L^{\infty}$ and strongly convergent, by Sobolev embedding theorem, in $L^{q}$ for any $q$. This is enough to show the strong convergence of $\Omega_k P_k$ to $\Omega P$ in $L^s$; altogether, we may pass to the strong limit in $L^s$ in the system \eqref{thm main:system}, showing that the equation
$$
dP+\Omega P=P\ast d\xi \qquad \text{in }\bb^n
$$
is satisfied in the sense of distributions.

The estimates \eqref{thm main:est mr} for $P$ and $\xi$ are then obvious.

\end{proof}

\begin{lemma}\label{lemma:existence2}
Let
\begin{align*}
\zeta & \colon \bb^n \to so(m)\otimes \Lambda^{n-2}\R^n \\
\lambda & \colon \bb^n \to so(m)\otimes\Lambda^1\R^n \\
Q & \colon \bb^n \to SO(m).
\end{align*}

Assume $\zeta$ belongs to the Morrey-Sobolev space $L_2^{p,n-2p+\alpha}$ and
$$
\lambda \in L^{2p,n-2p+\alpha}, \qquad d\lambda \in L^{p,n-2p+\alpha}
$$
There exist constants $\kappa =\kappa (p,n)>0$ and $\eta =\eta(p,n)>0$ such that if the following smallness conditions are satisfied
\begin{equation} \label{zeta estimate}
\begin{split}
\| d\zeta \|_{L^{2p,n-2p}}  \leq \kappa, \\
\| \lambda \|_{L^{2p,n-2p+\alpha}} + \| d \lambda \|_{L^{p,n-2p+\alpha}} \leq \eta,
\end{split}
\end{equation}
then there exists a solution $Q \in L_2^{p,n-2p+\alpha}$ of the equation
\begin{equation}\label{implicit2}
\ast d \ast\left(Q^{-1}dQ+Q^{-1}(\ast d\zeta+\lambda)Q\right)=0
\end{equation}
with $Q=\id$ on $\partial \bb^n$.

\end{lemma}

Note that \eqref{implicit2} implies, through Poincar\'e's lemma, that the term in parentheses is of the form $\ast d\tilde{\zeta}$, for some antisymmetric $(n-2)$-form $\tilde{\zeta}$.

\begin{proof}
Since we are interested in finding any $Q$ satisfying \eqref{implicit2}, we shall look for one of the form $e^u$, where
$$
u \colon \bb^n \to so(m), \quad u\in L_2^{p,n-2p+\alpha}
$$

Also, we need the boundary condition on $Q$ to hold, so we ask that $u$ has zero boundary values.

The equation \eqref{implicit2}, together with the boundary condition, can be rewritten as
\begin{equation} \label{system u}
\begin{split}
\Delta u &= \ast d \ast
\left(du - e^{-u} de^u - e^{-u}(\ast d \zeta + \lambda) e^u \right), \\
u &= 0 \quad \text{on $\partial \bb^n$}.
\end{split}
\end{equation}

We follow the proof of Tao and Tian \cite{TaoTian}, setting up the iteration scheme
\begin{align} \label{it scheme}
\Delta u^{k+1} &= \ast d \ast F(u^k,\zeta, \lambda), \\
u^{k+1} &=  0 \quad \text{on $\partial \bb^n$}, \nonumber \\
u^0 &= 0, \nonumber
\end{align}
where
\begin{equation} \label{F rhs}
F(u,\zeta,\lambda) = du - e^{-u} de^u - e^{-u}(\ast d \zeta + \lambda) e^u.
\end{equation}

Some calculations need more subtle justification though, since we work in noncommutative setting.
We will show that  there exists $\delta>0$ such that in each step of the recurrence
$$
\text{if} \quad
\| u^k \|_{L_2^{p,n-2p+\alpha}} \leq \delta
\quad \text{then} \quad
\| u^{k+1} \|_{L_2^{p,n-2p+\alpha}} \leq \delta.
$$

\medskip

We start with an easy observation. Since
$$
L_2^{p, n-2p +\alpha} \hookrightarrow L_1^{p,n-p+\alpha} \hookrightarrow C^{0,\gamma} \quad \text{for some $\gamma$}
$$
and H\"older-continuous functions on the bounded domain $B$ are bounded, it follows that if $f \in L_1^{p,n-p+\alpha}(B)$  with $f=0$ on $\partial \Omega$, then

\begin{equation*}
\|f\|_\infty \leq [f]_{C^{0,\gamma}} \leq C_M \| f \|_{L_1^{p,n-p+\alpha}(B)} \leq
C_M \| f \|_{L_2^{p,n-2p+\alpha}(B)},
\end{equation*}
where $C_M$ is the constant from the Morrey--Sobolev Embedding Lemma (see Proposition \ref{morrey-sobolev}).

Therefore, for every $\beta$, whenever $\delta < \frac{\beta}{C_M}$ it holds
\begin{equation} \label{bounded by morrey}
\text{if} \quad \|f\|_{L_2^{p,n-2p+\alpha}(B)} < \delta, \quad
\text{then} \quad \| f \|_\infty \leq \beta.
\end{equation}

\medskip

Now we can start the induction, assuming
\begin{equation} \label{uk bounded}
\|u^k \|_{L_2^{p,n-2p+\alpha}(B)} < \delta < \text{an absolute constant (to be specified later)}.
\end{equation}
Although the value of $\beta$ will be fixed later, we may assume already that $\beta<\frac{1}{2}$.

We will show first that the following pointwise estimates hold
\begin{equation} \label{pointwise}
| \ast d \ast F(u^k,\zeta, \lambda)| \leq
C(n) \Big( |u^k| |D^2 u^k| + |du^k|^2 + | d \zeta || du^k | + | \lambda || du^k | + |d \lambda | \Big) .
\end{equation}
Indeed, let $E(u)=e^{-u}D\exp(u)$. Since $E$ is smooth, both $E$ and $DE$ are bounded and Lipschitz continuous on $\{|u|<\beta\}$, the same holds, obviously, for $\exp$ and $D\exp$. Then
$$
F(u^k,\zeta, \lambda)=\left(E(0)-E(u^k)\right) du^k+e^{-u^k}\left(\ast d\zeta +\lambda\right)e^{u^k},
$$
and
\begin{equation*}
\begin{split}
|\ast d \ast F(u^k,\zeta, \lambda)| &\leq |DE(u^k)||du^k|^2+|E(0)-E(u^k)||D^2 u^k|\\
&\quad +C(n)\left(|du^k|(|d\zeta|+|\lambda|)+|d\lambda|\right)\\
&\leq C(n) \Big( |u^k| |D^2 u^k| + |du^k|^2 + | d \zeta || du^k | + | \lambda || du^k | + |d \lambda | \Big)
\end{split}
\end{equation*}

\medskip

Passing from pointwise to $L^{p,n-2p+\alpha}$ estimates, using Lemma \ref{Morrey product} we obtain
\begin{equation*}
\begin{split}
\| \ast d \ast &F(u^k,\zeta, \lambda)\|_{L^{p,n-2p+\alpha}}\\
&\leq C_1(n,p) \Big( \|u^k\|_{L^\infty}\| u^k \|_{L_2^{p,n-2p+\alpha}} + \|d u^k\|^2_{L^{2p,n-2p+\alpha}}\\
& \qquad + (\| d \zeta \|_{L^{2p,n-2p}} + \| \lambda \|_{L^{2p,n-2p+\alpha}} )
\| u^k \|_{L_2^{p,n-2p+\alpha}} +  \| d\lambda \|_{L^{p,n-2p+\alpha}} \Big).
\end{split}
\end{equation*}
The smallness conditions \eqref{zeta estimate}, \eqref{bounded by morrey} and \eqref{uk bounded} then imply
\begin{equation*}
\| \ast d \ast F(u^k,\zeta, \lambda)\|_{L^{p,n-2p+\alpha}}
\leq C_2(n,p,C_M) \Big( \beta \delta + \delta^2 + (\kappa + \eta) \delta +
 \eta \Big),
\end{equation*}
Regularity estimates for linear elliptic systems (see \cite{giaquinta}) yield
\begin{equation*}
\|u^{k+1} \|_{L_2^{p,n-2p+\alpha}}
\leq
C_3(n,p,C_M) \Big( \beta \delta + \delta^2 + (\kappa + \eta) \delta +
 \eta \Big)
\end{equation*}

W.l.o.g. we may assume $C_3 >1$. Let us choose $\beta$, $\eta$ and $\kappa$ such that
$$
\beta < \frac{1}{4C_3},
$$
$$
\eta < \min\{\frac{\beta}{4C_3 C_M}, \frac{1}{16 C_3^2} \},
$$
and
$$
\kappa < \frac{3}{16C_3}.
$$

Now we set
$$
\delta = 4C_3 \eta < \min\{ \frac{\beta}{C_M}, \frac{1}{4 C_3} \}.
$$
Then, if
$$
\| u^{k} \|_{{L_2^{p,n-2p+\alpha}}} \leq \delta,
$$
we have
\begin{equation*}
\|u^{k+1} \|_{L_2^{p,n-2p+\alpha}} \leq \delta.
\end{equation*}

Now, let us apply the same scheme to the differences of $u^k$. We assume that $u,v\in L_2^{p,n-2p+\alpha}$, $\|u\|_{L_2^{p,n-2p+\alpha}}<\delta$, $\|v\|_{L_2^{p,n-2p+\alpha}}<\delta$ and $u,v=0$ on $\partial \bb^n$.
This, in particular, implies (by \eqref{bounded by morrey}), that $\|u\|_{L^\infty}$ and $\|v\|_{L^\infty}$ are at most $\beta$ and although $\beta$ is to be specified later, we assume as before that it is less than $\frac12$.
We have
\begin{equation*}
\begin{split}
F&(u,\zeta,\lambda)-F(v,\zeta,\lambda)\\
& =\big( d(u-v) - e^{-u}de^u+e^{-v} de^v \big) + \big( -e^{-u}(\ast d\zeta+\lambda)e^u+e^{-v}(\ast d\zeta+\lambda)e^v \big) \\
&=\Big(\left(E(u)-E(0)\right)d(v-u)+\left(E(v)-E(u)\right)dv\Big)\\
&\quad +\Big((e^{-v}-e^{-u})(\ast d\zeta+\lambda)e^u+e^{-v}(\ast d\zeta+\lambda)(e^v-e^u)\Big)\\
& = S_1 + S_2.
\end{split}
\end{equation*}
Next, to estimate $|\ast d \ast F(u,\zeta,\lambda)-F(v,\zeta,\lambda)|$, we shall estimate separately $|\ast d \ast S_1|$ and $|\ast d \ast S_2|$, to avoid multi-line calculations.

Using as before boundedness and Lipschitz continuity of $E$, $DE$, $\exp$ and $D\exp$ on $\{|u|\leq\beta\}$ and keeping in mind that $\ast d \ast (\ast d \zeta)=0$ we get
\begin{equation}\label{S1 pointwise}
\begin{split}
|\ast d \ast S_1|&\leq |DE(u)||du||d(u-v)|+|E(u)-E(0)||D^2(u-v)|\\
&\quad +|E(u)-E(v)||D^2v|+|DE(v)dv-DE(u)du||dv|\\
&\leq C\left(|du||d(u-v)|+|u||D^2(u-v)| +|u-v||D^2 v|\right)\\
&\quad +|(DE(v)-DE(u))dv+DE(u)(dv-du)||dv|\\
&\leq C\left(|u-v|(|D^2v|+|dv|^2)+|d(u-v)|(|du|+|dv|)+|D^2(u-v)||u|\right)
\end{split}
\end{equation}
and
\begin{equation}\label{S2 pointwise}
\begin{split}
|\ast d \ast S_2|&\leq |D\exp(-u)du -D\exp(-v)dv|(|d \zeta|+|\lambda|)|e^u|\\
&\quad +|e^{-u}-e^{-v}||d\lambda||e^u|+|e^{-u}-e^{-v}|(|d\zeta|+|\lambda|)|D\exp(u)||du|\\
&\quad +|D\exp(-v)||dv|(|d\zeta|+|\lambda|)|e^u-e^v|+|e^{-v}||d\lambda||e^u-e^v|\\
&\quad +|e^{-v}|(|d\zeta|+|\lambda|)|D\exp(u) du-D\exp(v) dv|\\
&\leq C\big(|u-v|(|d\lambda| + (|du|+|dv|)(|d\zeta|+|\lambda|))\\
& \quad + |d(u-v)|(|d\zeta|+|\lambda|)\big).
\end{split}
\end{equation}

Altogether, adding up the estimates \eqref{S1 pointwise} and \eqref{S2 pointwise}, we obtain the following pointwise estimate
\begin{equation}
\begin{split}
  |\ast d \ast (F(u,\zeta,\lambda)&-F(v,\zeta,\lambda))|\\
  &\leq C\Big( |u-v|\big(|D^2v|+|dv|^2+|d\lambda|+(|du|+|dv|)(|d\zeta|+|\lambda|)\big)\\
   	 &\quad +|d(u-v)|\big(|du|+|dv|+|d\zeta|+|\lambda|\big)+|D^2(u-v)||u|\Big).
\end{split}
\end{equation}
with $C$ a universal constant.

Passing from the pointwise to $L^{p,n-2p+\alpha}$ estimates, using repeatedly \eqref{bounded by morrey}, H\"older's inequality, Lemma \ref{Morrey product} and keeping  in mind all smallness conditions,~i.e.

\renewcommand{\arraystretch}{1.5}
\hspace{2cm}\begin{tabular}{rr}
&$\displaystyle\|d\zeta\|_{L^{2p,n-2p}} \leq \kappa$,\\
\multicolumn{2}{r}{$\displaystyle\| \lambda \|_{L^{2p,n-2p+\alpha}} + \| d \lambda \|_{L^{p,n-2p+\alpha}}  \leq \eta$,}\\
$\displaystyle\| u \|_{{L_2^{p,n-2p+\alpha}}}  \leq \delta$,& $\displaystyle\| v \|_{{L_2^{p,n-2p+\alpha}}}  \leq \delta$, \\
$\displaystyle\|u\|_{L^\infty} < \beta$,&\qquad \qquad 
$\displaystyle\|v\|_{L^\infty} < \beta$,
\end{tabular}\\
\renewcommand{\arraystretch}{1}

we obtain
\begin{equation*}
\begin{split}
 \| \ast d \ast & (F(u,\zeta,\lambda) - F(v,\zeta,\lambda)) \|_{L^{p,n-2p+\alpha}}\\
 &\leq C(n,p)(\delta^2+\delta(3+2\kappa+2\eta)+\kappa+2\eta+\beta)\|u-v\|_{L_2^{p,n-2p+\alpha}}.
\end{split}
\end{equation*}

If we denote $H(u^k) = u^{k+1}$, where $u^{k+1}$ is a solution to \eqref{it scheme}, then
\begin{multline*}
 \| H(u) - H(v)  \|_{L_2^{p,n-2p+\alpha}}  \\
 \leq C_E C(n,p) \| u-v \|_{L_2^{p,n-2p+\alpha}}  \big(  \delta^2 + \delta + \delta (\kappa + \eta) + \beta + \kappa + \eta \big),
\end{multline*}
where $C_E$ is an absolute constant from elliptic estimates. Now, in order to show that $H$ is a contraction, we choose $\beta$  and $\kappa$ and $\eta$ sufficiently small.  The choice of $\beta$ and $\eta$  results in the choice of $\delta$. Therefore, by the Banach fixed point theorem, the iteration scheme \eqref{it scheme} converges and we obtain the desired solution of the system \eqref{system u}.

\end{proof}

The rest of the proof mimics the proof in the Sobolev case, and the Lemmata \ref{a priori2} and \ref{lemma:openness2} are direct counterparts of Lemmata~\ref{a priori} and~\ref{lemma:openness} from Section~\ref{sec:Sobolev}.

\begin{lemma}\label{a priori2}
Suppose $p < n/2$. There exists $\kappa=\kappa(p,n)$ with the following property: suppose that for  $\Omega\in V^\alpha_\epsilon $ there exist  $P$ and $\xi $ in $L_2^{p,n-2p + \alpha}(\bb^n)$ satisfying the system \eqref{main2} and additionally the estimate
\begin{equation}\label{kappa estimate mr}
\|dP\|_{L^{2p,n-2p}(\bb^n)}+\|d\xi \|_{L^{2p,n-2p}(\bb^n)}<\kappa.
\end{equation}
Then the estimates \eqref{thm main:est mr} hold.
\end{lemma}

\begin{remark}
Note that the fact that $dP$, $d\xi \in L^{2p,n-2p}$ follows from \eqref{kwadrat w Morreyach}.
\end{remark}

\begin{proof}[Proof of Lemma \ref{a priori2}]

We have in the ball $\bb^n$
\begin{equation} \label{laplacian xi}
\begin{split}
\Delta \xi&=(\ast d\ast d+d\ast d\ast )\xi=\ast d(\ast d\xi)\\
&=\ast d(P^{-1}dP+P^{-1}\Omega P)\\
&=\ast(dP^{-1}\wedge dP)+\ast d(P^{-1}\Omega P).
\end{split}
\end{equation}

Let $B = B_r(x_0)$ be a fixed ball. On the set $B \cap \bb^n$ we split $\xi$ into a sum of two functions
$\xi = u + v$, satisfying
\begin{align*}
\Delta u &= \ast (dP^{-1}\wedge dP) \quad \text{on $B \cap \bb^n$},\\
u &= 0 \quad  \text{on $\partial (B \cap \bb^n) $}
\end{align*}
(we may assume $\mathrm{supp}\, u \subset B \cap \bb^n$) and
\begin{align*}
\Delta v &= \ast d(P^{-1} \Omega P) \quad \text{on $\bb^n$},\\
v &= 0 \quad  \text{on $\partial \bb^n$}.
\end{align*}
Thus $\xi = u+v$ on $B\cap \bb^n$.

For $q=p/(p-1)$ we have
\begin{equation*}
\| du \|_{L^p(B\cap \bb^n)} \leq C_q \sup_{\|d \phi\|_{L^q}\leq 1}\int_{B \cap \bb^n} d u \cdot d\phi,
\end{equation*}
where $\phi$ is a smooth, compactly supported (in particular with null boundary values)  $(n-2)$-form on $B \cap \bb^n$ (c.f. the proof of Lemma \ref{a priori} ).
Denote by $\bar{P}$ the mean value of $P$ over $\bb^n$: $\bar{P}=\mvint_{\bb^n}P$. For any $\phi$ as above,
\begin{equation*}
\begin{split}
\int_{B \cap \bb^n} du \cdot d\phi&=-\int_{B \cap \bb^n}  \Delta u \cdot \phi\\
&=-\int_{\bb^n} \ast \left(dP^{-1}\wedge d(P-\bar{P})\right) \cdot \phi\\
&=-\int_{\bb^n} dP^{-1}(P-\bar{P})\wedge d\phi \\
&\leq C \| dP^{-1} \wedge d\phi\|_{\mathcal{H}^1(\bb^n)} \|P-\bar{P}\|_{BMO(\bb^n)}\\
&= C \| dP^{-1} \wedge d\phi\|_{\mathcal{H}^1(B \cap \bb^n)} \|P-\bar{P}\|_{BMO(\bb^n)}\\
&\leq C \| dP \|_{L^p(B \cap \bb^n)}  \| d\phi \|_{L^q(B \cap \bb^n)}  \| d P \|_{L^{p,n-p}(\bb^n)}
\end{split}
\end{equation*}
with the constants $C$ (possibly different in every line) dependent only on $n$ and $p$. Note that since $P$ is an orthogonal matrix, $|P|=|P^{-1}|=1$, $|dP^{-1}|=|dP|$. In the estimate above we use the Coifman-Lions-Meyer-Semmes div-curl inequality (\cite{CLMS}) and later the inclusion
$$
L_1^{p,n-p}(\bb^n) \hookrightarrow BMO(\bb^n).
$$

We have then
\begin{align*}
\| du \|_{L^p(B\cap \bb^n)} &\leq C \| dP \|_{L^p(B \cap \bb^n)} \| d P \|_{L^{p,n-p}(\bb^n)}\\
& \leq C \| dP \|_{L^p(B \cap \bb^n)} \| d P \|_{L^{2p,n-2p}(\bb^n)}\\
& \leq C \kappa \| dP \|_{L^p(B \cap \bb^n)}
\end{align*}
due to the smallness assumption \eqref{kappa estimate mr}. Therefore
\begin{equation} \label{est u}
\| du \|_{L^{p,n-p+\alpha} (\bb^n)} \leq C \kappa \| dP \|_{L^{p,n-p+\alpha}(\bb^n)}
\end{equation}
for $\alpha \geq 0$.

To estimate $v$, where
\begin{align*}
\Delta v &= \ast d(P^{-1}\Omega P) \quad \text{on $B \cap \bb^n$},\\
v &= 0 \quad  \text{on $(\partial \bb^n) \cap B, $}
\end{align*}
we use standard elliptic estimates, obtaining
\begin{equation} \label{est v}
\| dv \|_{L^{p,n-p+\alpha} (\bb^n)} \leq C \| \Omega \|_{L^{p,n-p+\alpha}(\bb^n)} .
\end{equation}

Combining \eqref{est u} and \eqref{est v} we conclude with
\begin{equation} \label{est dxi}
\| d\xi \|_{L^{p,n-p+\alpha} (\bb^n)} \leq C \kappa \| dP \|_{L^{p,n-p+\alpha}(\bb^n)}  + C \| \Omega \|_{L^{p,n-p+\alpha}(\bb^n)}.
\end{equation}

On the other hand, taking $L^p$ norms of both sides of the equation
\begin{equation}\label{eq for dP2}
dP=P\ast d\xi+ \Omega P \qquad \text{in $\bb^n$}
\end{equation}
(c.f \eqref{main2}) gives
\begin{equation*}
\|dP\|_{L^p(B \cap \bb^n) }\leq \|d\xi \|_{L^p(B \cap \bb^n)}+\|\Omega \|_{L^p(B \cap \bb^n)}
\end{equation*}
and thus
\begin{equation} \label{est dP}
\|dP\|_{L^{p, n-p+\alpha}(\bb^n) }\leq \|d\xi \|_{L^{p, n-p + \alpha}(\bb^n)}+\|\Omega \|_{L^{p, n-p+\alpha}(\bb^n)}.
\end{equation}
Putting \eqref{est dxi} and \eqref{est dP} together we get, for $\alpha \geq 0$,
\begin{equation*}
\|d\xi\|_{L^{p,n-p+\alpha}(\bb^n)}+\|dP\|_{L^{p,n-p+\alpha}(\bb^n)}
\leq
C\kappa \|dP\|_{L^{p,n-p+\alpha}(\bb^n)}+C \|\Omega \|_{L^{p,n-p+\alpha}(\bb^n)}.
\end{equation*}
 Taking $\kappa$ small enough we conclude the estimates \eqref{thm main:est1 mr},  \eqref{thm main:est2 mr} hold, i.e.
$$
\|d\xi\|_{L^{p,n-p+\alpha}(\bb^n)}+\|dP\|_{L^{p,n-p+\alpha}(\bb^n)}
\leq
C \|\Omega \|_{L^{2p,n-2p+\alpha}(\bb^n)}
$$
for $\alpha \geq 0$.

\medskip

To show the estimate \eqref{thm main:est3 mr}, taking $\ast d \ast$ of both sides of \eqref{eq for dP2}, we see that
\begin{equation*}
\begin{split}
\Delta P&=\ast d\ast d P=\ast d\ast (P\ast d\xi+ \Omega P)\\
&=\ast(dP\wedge d\xi)+\ast d \ast (\Omega P),
\end{split}
\end{equation*}
and from \eqref{laplacian xi} we obtain
\begin{equation*}
\Delta \xi=\ast(dP^{-1}\wedge dP)+\ast d(P^{-1}\Omega P).
\end{equation*}
Therefore, proceeding as in the proof of the estimates \eqref{dP in sobolev} and \eqref{dxi in sobolev},  we obtain
\begin{equation*}
\| \Delta P\|_{L^{p,n-2p}}
\leq
\|dP\|_{L^{2p,n-2p}} \| d\xi \|_{L^{2p,n-2p}} + \| d\Omega \|_{L^{p,n-2p}}
	+ \| \Omega \|_{L^{2p,n-2p}} \| dP \|_{L^{2p,n-2p}}\\
\end{equation*}
and
\begin{equation*}
\| \Delta \xi\|_{L^{p,n-2p}}
\leq  \|dP\|^2_{L^{2p,n-2p}} + 2 \|dP\|_{L^{2p,n-2p}} \| \Omega \|_{L^{2p,n-2p}}
	+ \| d\Omega \|_{L^{p,n-2p}}
\end{equation*}
Applying \eqref{kappa estimate mr}, we obtain
\begin{equation*}
\| \Delta P\|_{L^{p,n-2p}} + \| \Delta \xi\|_{L^{p,n-2p}}
\leq
2 \| d\Omega \|_{L^{p,n-2p}} + 3\kappa  \| \Omega \|_{L^{2p,n-2p}}
\end{equation*}
which proves the estimate \eqref{thm main:est3 mr}.

Observe that the above inequality is a consequence of the equation \eqref{eq for dP2} and the H\"older inequality only, so the estimate holds in any Morrey space $L^{p,\gamma}$ with $\gamma > n-2p$, in which both sides of the inequality are finite.

\end{proof}

\begin{lemma}\label{lemma:openness2}
The set $U_\epsilon^\alpha$ is, for $\epsilon $ sufficiently small, open in $V_\epsilon^\alpha $.
\end{lemma}

\begin{proof}

Choose $\Omega_o\in U_\epsilon^\alpha$ and let $P_o$ and $\xi_o$ be the orthogonal transformation and antisymmetric $(n-2)$-form associated with $\Omega_o$, so that Lemma \ref{thm main mr} holds for $\Omega_o$, $P_o$ and $\xi_o$.

Take now $\Omega\in V_\epsilon^\alpha$ close to $\Omega_o$ in $L^{2p,n-2p+\alpha} \cap L_1^{p, n-2p + \alpha}$: we ask that for $\lambda =P_o^{-1}(\Omega -\Omega_o)P_o$ we have
$$
\| \lambda\|_{L^{2p,n-2p+\alpha}}+\|\lambda \|_{L_1^{p,n-2p+\alpha}}<\eta
$$
 (the conjugation with $P_o\in L_2^{p,n-2p+\alpha}$ is continuous in $L_1^{p,n-2p+\alpha}$). Applying Lemma \ref{lemma:existence2} with $\zeta =\xi_o$ we find $Q\in L_2^{p,n-2p+\alpha}(\bb^n,SO(m))$ such that
\begin{equation}\label{open:1a}
\ast d \ast\left(Q^{-1}dQ+Q^{-1}(\ast d\xi_o+P_o^{-1}(\Omega -\Omega_o)P_o)Q\right)=0.
\end{equation}
Setting $P=P_o Q$ we see that \eqref{open:1a} reduces to
$$
\ast d \ast (P^{-1}dP+P^{-1}\Omega  P)=0.
$$
By Poincar\'e's lemma this implies that $P^{-1}dP+P^{-1}\Omega  P$  is a coexact form, i.e. there exists an antisymmetric  $(n-2)$-form $\xi$ such that
\begin{equation}\label{last2}
\ast d\xi=P^{-1}dP+P^{-1}\Omega  P,
\end{equation}
thus $P$ and $\xi$ give Uhlenbeck's decomposition of $\Omega$.

Note that $Q$ and $P_o \in L_2^{p,n-2p+\alpha}\cap L^\infty$ imply that $P\in L_2^{p,n-2p+\alpha}$.
By Hodge decomposition theorem we can choose $\xi$ to be coclosed ($d\ast\xi=0$ on $\bb^n$) and to have zero boundary values ($\xi|_{\partial \bb^n}=0$). Finally, the right hand side of \eqref{last2} is in $L_1^{p,n-2p+\alpha}$, which gives $\xi\in L_2^{p,n-2p+\alpha}$.

What remains to prove is that $P$, $\xi$ and $\Omega $ satisfy the estimates \eqref{thm main:est mr}. Observe that if $\|\Omega -\Omega_o\|_{L_1^{p,n-2p+\alpha}}$ is small enough, then by continuity of the mapping $\lambda \mapsto u_\lambda$ so is  $\|P-P_o\|_{L_1^{p,n-2p+\alpha}}$ and $\|\xi-\xi_o\|_{L_1^{p,n-2p+\alpha}}$; choosing $\eta$ (which measures the distance $\|\Omega -\Omega_o\|_{L_1^{p,n-2p+\alpha}}$) sufficiently small we get
$$
\|P-P_o\|_{L_1^{p,n-2p+\alpha}}+\|\xi-\xi_o\|_{L_1^{p,n-2p+\alpha}}<\epsilon.
$$
We also know that
$$
\|d\xi_o\|_{L^{2p,n-2p}}+\|dP_o\|_{L^{2p,n-2p}}\leq C\|\Omega_o\|_{L^{2p,n-2p}}\leq C\epsilon
$$
(this follows from $\Omega_o\in U_\epsilon$).

Taking $\epsilon$ sufficiently small we may ensure that
\begin{equation*}
\begin{split}
&\|d\xi\|_{L^{2p,n-2p}}+\|dP\|_{L^{2p,n-2p}}\\
&\leq\|d\xi-d\xi_o\|_{L^{2p,n-2p}}+\|P-P_o\|_{L^{2p,n-2p}}+\|d\xi_o\|_{L^{2p,n-2p}}+\|dP_o\|_{L^{2p,n-2p}}\\
&<(C+1)\epsilon<\kappa,
\end{split}
\end{equation*}
with $\kappa$ as in Lemma \ref{a priori2}. Applying this lemma we show that the estimates \eqref{thm main:est mr} hold.

Altogether, $\Omega\in U_\epsilon^\alpha$, which proves the openness of $U_\epsilon^\alpha$.\\
\end{proof}

\begin{proof}[Proof of Theorem \ref{u-m}]
The proof mimics, in a way, the passage from Theorem \ref{thm main} to Corollary \ref{corr}, i.e. from the Uhlenbeck decomposition in $W^{1,p}$ for $p>n/2$ to the decomposition for $p=n/2$. There, we could simply argue by approximation. In the Morrey space setting, however, neither $L_1^{p,n-p+\alpha}$ embeds densely in $L_1^{p,n-p}$, nor $L^{2p,n-2p+\alpha}$ does into $L^{2p,n-2p}$.

 However (cf. \cite{RiviereStruwe}, proof of Lemma 3.1), one can easily prove that if $(\phi_r)$ is a standard mollifier and $f\in L^{q,n-q}(\bb)$, $q\geq 1$, then on any ball $B=B(x,\rho)$ such that $2B=B(x,2\rho)\subset\bb^n$ we have, for $r<\rho$, that $\|f\ast \phi_r\|_{L^{q,n-q}(B)}\leq \|f\|_{L^{q,n-q}(2B)}$. Reasoning like in the proof of Meyers-Serrin's theorem and using a suitable decomposition of unity we can show then that there exists a sequence $f_k\in C^\infty(\bb)$ convergent to $f$ in $L^q$ (and in any other appropriate Lebesgue and Sobolev norm) such that $\|f_k\|_{L^{q,n-q}(\bb)}\leq C(n)\|f\|_{L^{q,n-q}(\bb)}$.

 We thus proceed as follows: we approximate $\Omega$ in $W^{1,p}$ by a sequence of smooth $\Omega_k$ such that for all $k$

 \begin{equation} \label{Omega k controlled}
 \| \Omega_k \|_{L^{2p,n-2p}(\bb)} \leq C(n) \| \Omega \|_{L^{2p,n-2p}(\bb)}.
 \end{equation}

 Assuming that $\epsilon$ in the condition $\|\Omega\|_{L^{2p,n-2p}}<\epsilon$ is taken small enough we can ensure, through \eqref{Omega k controlled}, that all $\Omega_k$ satisfy the analogous smallness condition in Lemma \ref{thm main mr}. This provides us with sequences of $P_k$ and $\xi_k$ that give the Uhlenbeck decomposition for $\Omega_k$, together with the uniform estimate
$$
\|d\xi_k \|_{L^{p, n-p}}+\|dP_k\|_{L^{p,n-p}} \leq  C(n,m)\|\Omega \|_{L^{2p, n-2p}} <C \epsilon.
$$
Then we proceed as in the proof of closedness of $U_\epsilon^\alpha$ in Lemma \ref{lemma:closedness2}, obtaining convergent subsequences of $P_k$ and $\xi_k$. As $\Omega_k$ are smooth, they satisfy the assumptions of Lemma \ref{thm main mr}, which gives us the estimates \eqref{thm main:est mr}
\begin{align*}
\|d\xi_k \|_{L^{p, n-p}}+\|dP_k\|_{L^{p,n-p}} &\leq  C(n,m)\|\Omega_k \|_{L^{2p, n-2p}},\\
\|\Delta \xi_k \|_{L^{p, n-2p}}+\| \Delta P_k\|_{L^{p,n-2p}} &\leq
C(n,m)\left( \|\Omega_k \|_{L^{2p, n-2p}} + \|\nabla \Omega_k \|_{L^{p, n-2p}}\right).
\end{align*}
The sequences  $P_k$ and $\xi_k$  converge in $L_2^{p,n-2p}$ to appropriate elements of a decomposition of $\Omega$ (this follows from the equation they satisfy). Thus, the above estimates, together with \eqref{Omega k controlled}, yield the desired estimates  \eqref{thm main:est mr} for~$\Omega$. This completes the proof of Theorem~\ref{u-m}.

\end{proof}


\section{Uhlenbeck's decomposition and conformal matrices}\label{sec: conf}

A natural extension of the orthogonal gauge group $SO(m)$ is the conformal group $CO_+(m)$. The interest in this group has deep roots in complex analysis, in particular in the studies related to Liouville's Theorem (see \cite{IwaniecMartin} for a detailed exposition). This is a non-compact group, defined as
$$
CO_+(m)=\{\lambda P ~~:~~\lambda\in \R_+,\, P\in SO(m)\}.
$$
Clearly, $S\in CO_+(m)$ iff $SS^T=\lambda^2 \id$, where by $\id$ we denote the $m\times m$ identity matrix.

The tangent space at $\id$ to $CO_+(m)$, which we denote by $TCO_+(m)$, is given as
$$
TCO_+=\{K\in \mathcal{M}^{m\times m}~~:~~K+K^T=\frac{2 Tr(K)}{m}\otimes \id\},
$$
or, equivalently,
$$
TCO_+=\{A+\mu\id~~:~~A\in so(m),\, \mu\in \R\},
$$
see e.g. \cite{FaracoZhong}.

Our objective is to prove an analogue of Theorem~\ref{thm main} for the conformal gauge group, i.e. Theorem~\ref{thm:conformal}.
\begin{theorem*}
Let $\frac{n}{2}<p<n$. There exists $\epsilon >0$ such that for any $\Omega \in W^{1,p}(\bb^n, TCO_+(m)\otimes \Lambda^1\R^n)$ such that $\|\Omega \|_{L^n}<\epsilon $ there exist $S:\bb^n\to CO_+(m)$ satisfying $\ln |S|\in W^{2,p}(\bb^n)$, $S/|S|\in W^{2,p}(\bb^n, SO(m))$ and $\zeta \in W^{2,p}(\bb^n, TCO_+(m)\otimes \Lambda^{n-2}\R^n)$ such that
\begin{equation}\label{eq:thm2}
\left\{
\begin{aligned}
S^{-1}dS+S^{-1}\Omega S&=\ast d\zeta &&\text{on }\bb^n,\\
d\ast\zeta&=0 &&\text{on }\bb^n,\\
\zeta&=0 &&\text{on } \partial \bb^n;
\end{aligned}
\right.
\end{equation}
and such that
\begin{subequations}\label{thm2:est}
\begin{align}
\|d\zeta\|_{W^{1,p}}+\|d(S/|S|)\|_{W^{1,p}}+\|d\ln|S|\|_{W^{1,p}}&\leq C(n,m)\|\Omega \|_{W^{1,p}}\label{thm2:est1}\\
\|d\zeta \|_{L^p}+\|d(S/|S|)\|_{L^p}+\|d\ln |S|\|_{L^p}&\leq C(n,m)\|\Omega \|_{L^p},\label{thm2:est2}\\
\|d\zeta \|_{L^n}+\|d(S/|S|)\|_{L^n}+\|d\ln |S|\|_{L^n}&\leq C(n,m)\|\Omega \|_{L^n}.\label{thm2:est3}
\end{align}
\end{subequations}
\end{theorem*}
The integrability conditions on $\ln|S|$ should be understood as (rather weak) integrability conditions both on $S$ and $S^{-1}$.
We should also note that if $S$ satisfies the above theorem, so does $tS$ for any non-zero constant $t$.
\begin{proof}
We shall construct $S:\bb^n\to CO_+(m)$  satisfying the above conditions. Let us first fix some notation:

We shall write $S=\lambda P$, where $\lambda =|S|\in \R_+$ and $P=S/|S|\in SO(m)$, we also decompose $\Omega$ into its antisymmetric and diagonal part:
$$
\Omega=A+\frac{Tr(\Omega)}{m}\otimes\id
$$
with $A\in W^{1,p}(\bb^n, so(m)\otimes \Lambda^1\R^n)$.

Let $\Omega^S=S^{-1}dS+S^{-1}\Omega S$; likewise $\Omega^P=P^{-1}dP+P^{-1}\Omega P$ and $A^P=P^{-1}dP+P^{-1}A P$.

We have
\begin{equation*}
\begin{split}
\Omega^S&=\lambda^{-1}P^{-1}(d\lambda \otimes P+\lambda dP+P^{-1}\Omega P)\\
&=d\ln\lambda \otimes\id+\Omega^P.
\end{split}
\end{equation*}
Decomposing $\Omega $ we have
\begin{equation*}
\begin{split}
\Omega^P&=P^{-1}dP+P^{-1}\left(A+\frac{Tr(\Omega )}{m}\otimes\id\right) P \\
&=A^P+\frac{Tr(\Omega )}{m}\otimes\id,
\end{split}
\end{equation*}
thus
\begin{equation*}
\Omega^S=A^P+\left(d\ln\lambda +\frac{Tr(\Omega )}{m}\right)\otimes \id.
\end{equation*}
Clearly, $A$ satisfies all the assumptions on $\Omega$ in Theorem \ref{thm main}, we can thus find
$P\in W^{2,p}(\bb^n,SO(m))$ and $\xi \in W^{2,p}(\bb^n, so(m)\otimes \Lambda^{n-2}\R^n)$ such that
\begin{equation}
\left\{
\begin{aligned}
A^P=P^{-1}dP+P^{-1}A P&=\ast d\xi &&\text{on }\bb^n,\\
d\ast\xi&=0 &&\text{on }\bb^n,\\
\xi&=0 &&\text{on } \partial \bb^n;
\end{aligned}
\right.
\end{equation}
and such that
\begin{subequations}\label{subeq:thm2}
\begin{align}
\|d\xi\|_{W^{1,p}}+\|dP\|_{W^{1,p}}&\leq C(n,m)\|A \|_{W^{1,p}}\leq C(n,m)\|\Omega \|_{W^{1,p}}\\
\|d\xi \|_{L^p}+\|dP\|_{L^p}&\leq C(n,m)\|A \|_{L^p}\leq C(n,m)\|\Omega \|_{L^p},\\
\|d\xi \|_{L^n}+\|dP\|_{L^n}&\leq C(n,m)\|A \|_{L^n}\leq C(n,m)\|\Omega \|_{L^n}.
\end{align}
\end{subequations}
By Hodge decomposition we can find $\alpha \in W^{2,p}(\bb^n)$ and $\beta \in W^{2,p}(\bb^n, \Lambda^{n-2}\R^n)$ such that
$$
\frac{1}{m}Tr(\Omega )=d\alpha +\ast d\beta
$$
with $\beta|_{\partial \bb^n}=0$ and $\|d\alpha\|_W^{1,p}+\|d\beta\|_{W^{1,p}}\leq\|\Omega\|_{W^{1,p}}$.

This shows that if $\lambda $ is such that $d\ln \lambda =-d\alpha $, then $S=\lambda P$ and $\zeta =\xi+\beta\otimes\id$ satisfy \eqref{eq:thm2}. The estimates \eqref{thm2:est} follow immediately from the estimates on Hodge decomposition and from \eqref{subeq:thm2}.
\end{proof}

The above theorem is rather simple, but it provides a new interpretation to gradient-like terms $df\otimes \id$ in nonlinear systems -- we can incorporate them in antisymmetric expressions and perform Uhlenbeck's decomposition on the resulting $TCO_+$ matrix of differential forms instead of dealing with both kinds of terms separately.

\section*{Acknowledgements}
The Authors would like to thank the Institute of Mathematics of the Univeristy of Jyv\"askyl\"a for their hospitality and to Prof. Andreas Gastel for his very useful comments and remarks.


\end{document}